\documentclass{amsart}
\usepackage[utf8]{inputenc}
\usepackage{amsmath}
\usepackage{amsfonts}
\usepackage{amssymb}
\usepackage{amsthm}
\usepackage{amscd}
\usepackage{float}
\usepackage{tikz}
\usepackage{graphicx}
\usepackage[colorlinks=true]{hyperref}
\hypersetup{urlcolor=blue, citecolor=red}
\usepackage{hyperref}

\newtheorem{theorem}{Theorem}[section]
\newtheorem{corollary}[theorem]{Corollary}

\newtheorem{lemma}[theorem]{Lemma}
\newtheorem{proposition}[theorem]{Proposition}

\theoremstyle{definition}

\newtheorem{remark}[theorem]{Remark}

\newcommand{\R}{\mathbb{R}}
\newcommand{\N}{\mathbb{N}}
\newcommand{\C}{\mathbb{C}}

\newcommand{\Z}{\mathbb{Z}}

\makeatletter
\@namedef{subjclassname@2020}{%
  \textup{2020} Mathematics Subject Classification}
\makeatother

\begin{document}

\title[Generalized square function estimates]{Generalized square function estimates for curves and their conical extensions}
\author{Robert Schippa}

\begin{abstract}
We show sharp square function estimates for curves in the plane whose curvature degenerates at a point and estimates sharp up to endpoints for cones over these curves. To this end, for curves of finite type we extend the classical C\'ordoba--Fefferman biorthogonality. For cones over degenerate curves, we analyze wave envelope estimates via High-Low-decomposition. The arguments are subsequently extended to the cone over the complex parabola.
\end{abstract}

\maketitle

\section{Introduction}

We consider square function estimates for curves with degeneracy and for their conical extensions. Let $\Gamma$ be a non-degenerate curve $\Gamma = \{(\xi,h(\xi)): \xi \in (-1,1) \}$ with $h \in C^2(-1,1)$, $h(0) = h'(0)=0$, and $h''(\xi) \sim 1$. Let $F \in \mathcal{S}(\R^2)$ with $\text{supp}(\hat{F}) \subseteq \mathcal{N}_\delta(\Gamma)$ with $\mathcal{N}_\delta(\Gamma)$ denoting the $\delta$-neighbourhood of $\Gamma$. Let $\Theta_\delta$ denote a covering of $\Gamma$ with rectangles $\theta$ of size $C \delta^{\frac{1}{2}} \times C \delta$, the short side pointing into normal direction, and the long side pointing into tangential direction. $C$ is chosen large enough such that $\theta \in \Theta_{\delta}$ cover $\mathcal{N}_{\delta}(\Gamma)$. The classical C\'ordoba--Fefferman square function
\cite{Cordoba1979,Cordoba1982,Fefferman1973} estimate reads
\begin{equation}
\label{eq:ClassicalCF}
\| F \|_{L^4(\R^2)} \lesssim \big\| \big( \sum_{\theta \in \Theta_\delta} |F_\theta|^2 \big)^{\frac{1}{2}} \big\|_{L^4(\R^2)}.
\end{equation}

Above $F_\theta$ denotes a mollified Fourier projection to $\theta$, which comprises a smooth partition of unity on $\mathcal{N}_\delta(\Gamma)$.
Whereas the proof consists of a simple geometric observation, this estimate is of immense significance for estimating Fourier multipliers and exemplarily yields (in a well-known combination with maximal function estimates) the sharp estimate for Bochner--Riesz multipliers in two dimensions, as pointed out in the aforementioned references. We refer to \cite{Stein1993} for further reading.

\smallskip

Here we shall look into variants for degenerate curves. We shall analyze in detail the model case: $\gamma_k = \{(\xi,\xi^k) : \xi \in [-1,1] \}$ for $k \in \N_{\geq 2}$. For $k=2$ this is a non-degenerate curve and the classical estimate applies. For $k \geq 3$, the curvature degenerates at the origin and the decomposition of $\gamma_k$ into rectangles of size $\delta^{\frac{1}{2}} \times \delta$ does not seem to be appropriate anymore. A decomposition into longer intervals of length $\delta^{\frac{1}{k}}$ upon projection to the first coordinate was pointed out by Biggs--Brandes--Hughes \cite{BiggsBrandesHughes2022}; see also Gressman \emph{et al.} \cite{GressmanGuoPierceRoosYung2021} highlighting the connection between counting and square function estimates and references therein. However, the $\delta^{\frac{1}{k}}$-intervals see some part of the curvature away from the origin, so the curve does no longer fit into rectangles of size $\delta^{\frac{1}{k}} \times \delta$.

Indeed, the rectangles $\theta$ occuring in the square function estimate in \eqref{eq:ClassicalCF} linearize the curve on the largest possible scale with an error of size $\delta$. This will be the guiding principle for us.

Let $\Theta_{\delta,k}$ be a covering of $\gamma_k$ with rectangles as follows: One rectangle of size comparable to $C \delta^{\frac{1}{k}} \times C \delta$ centered at the origin, and for $|\xi| \gtrsim \delta^{\frac{1}{k}}$, choose a rectangle centered at $(\xi,\xi^k)$ of length $\delta^{\frac{1}{k}} / |\xi|^{\frac{k-2}{2}}$ into tangential direction and length $\delta$ into normal direction. $C \geq 1$ is chosen large enough such that $\Theta_{\delta,k}$ forms a covering of $\mathcal{N}_\delta(\gamma_k)$.  We show the following:
\begin{theorem}
\label{thm:DegenerateSquareFunctionEstimateCurves}
Let $k \in \N_{\geq 2}$, and $F \in \mathcal{S}(\R^2)$ with $\text{supp}(\hat{F}) \subseteq \mathcal{N}_{\delta}(\gamma_k)$. Then the following estimate holds:
\begin{equation}
\label{eq:DegenerateSquareFunctionEstimate}
\| F \|_{L^4(\R^2)} \lesssim \big\| \big( \sum_{\theta \in \Theta_{\delta,k}} |F_{\theta}|^2 \big)^{\frac{1}{2}} \big\|_{L^4(\R^2)}.
\end{equation}
\end{theorem}
In the proof we consider second order differences, by which already Biggs--Brandes--Hughes proved the estimate with a coarser decomposition mentioned above.
The present analysis can be regarded as extension of the argument to show the square function estimate with a decomposition into the canonical scale, which depends on the curvature.

In the next step we consider square function estimates for cones over these curves. For $k \in \N_{\geq 3}$ we define
\begin{equation*}
\mathcal{C} \gamma_k = \{ (\omega_1,\omega_2,\omega_3) \in \R^3 : |\omega_1| \leq 1, \; \omega_3 \in [1/2,1], \; \omega_2 = \omega_1^k/ \omega_3^{k-1} \}.
\end{equation*}
 Let $\Theta^{\mathcal{C}}_{\delta,k}$ denote the conical extension of the canonical covering of the curve $\{ (\omega_1,\omega_2) \in \R^2 : \omega_2 = \omega_1^k \}$. Roughly speaking, we take the covering of the $\delta$-neighbourhood into canonical rectangles linearizing the curve on the largest possible scale and then take its conical extension. The detailed definition will be provided in Section \ref{section:SquareFunctionsCones}.

\smallskip

We show the following:
\begin{theorem}
Let $\delta >0$, $k \in \N_{\geq 3}$, $F \in \mathcal{S}(\R^3)$ and $\text{supp}(\hat{F}) \subseteq \mathcal{N}_\delta(\mathcal{C} \gamma_k)$. Then the following square function estimate holds:
\begin{equation*}
\| F \|_{L^4(\R^3)} \lesssim_\varepsilon \delta^{-\varepsilon} \big\| \big( \sum_{\theta \in \Theta_{\delta,k}^{\mathcal{C}}} |F_{\theta}|^2 \big)^{\frac{1}{2}} \big\|_{L^4(\R^3)}.
\end{equation*}
\end{theorem}

This extends the square function estimate for non-degenerate cones due to Guth--Wang--Zhang \cite{GuthWangZhang2020}. We analyze Kakeya estimates for planks $\tilde{\theta} = \theta - \theta$ via a  High-Low-decomposition. After fixing a height $h \sim \sigma^2$, $\sigma \in [\delta^{\frac{1}{2}},1]$, by considering their overlap, the planks can be sorted into smaller centred planks. This corresponds to a sorting into larger sectors $\tau \supseteq \theta$. In case of the circular cone the resulting expressions are highly symmetric. In the generalized case we show overlap estimates via more perturbative arguments. It turns out that there is an additional overlap between regions, where the base curve has significant different curvature. So, we cannot recover the same Kakeya estimate like in case of the circular cone, but slicing into regions of approximately constant curvature seems necessary. This is different from the square function estimates for curves, where we can consider all the different curvatures simultaneously and still recover the estimate \eqref{eq:DegenerateSquareFunctionEstimate} on the canonical scale without loss.

\medskip

We remark that Gao \emph{et al.} \cite{GaoLiuMiaoXi2023} showed another stability result by extending the constant-coefficient analysis from \cite{GuthWangZhang2020} to variable coefficients. In this case, on small spatial scales the analysis for non-degenerate cones can be utilized. Then, by rescaling and using the self-similar structure of wave envelope estimates (see \cite{BeltranHickmanSogge2020} for a variant of this argument in the context of decoupling), the estimate can be propagated to large spatial scales.

\smallskip

For the circular cone, the square function estimate combined with maximal function estimates yielded the sharp range of local smoothing estimates with $\varepsilon$-derivatives for solutions to the wave equations (\cite{Bourgain1986,Sogge1991,MockenhauptSeegerSogge1992}). In a different direction, Maldague and Guth--Maldague recently proved sharp square function estimates for moment curves $t \mapsto (t,t^2,t^3,\ldots,t^k)$ of cubic \cite{Maldague2022} and higher order \cite{GuthMaldague2023}.
It is conceivable that arguments related to the present analysis allow for the proof of sharp square function estimates for curves of finite type $t \mapsto (t,t^{a_2},\ldots,t^{a_k})$, $1 < a_2 < a_3 \ldots < a_k$.
\smallskip

The ``local" arguments to show the central Kakeya estimate extend to the complex cone:
\begin{equation*}
\C \Gamma_2 = \{ (z,z^2/h,h) \in \C \times \C \times \R : \, |z| \leq 1, h \in [1/2,1] \},
\end{equation*}
which, by identifying $\C \equiv \R^2$, can be regarded as subset of $\R^5$. Let $\Theta_{\delta, \C}$ denote the conical extension of the ``complex" rectangles canonically covering the complex base curve $z \mapsto (z,z^2)$. The precise definitions are deferred to Section \ref{section:SquareFunctionComplexCone}.

\smallskip

Extending the multiscale analysis of Guth--Wang--Zhang \cite{GuthWangZhang2020} to the complex cone, we can show the following square function estimate:
\begin{theorem}
\label{thm:SquareFunctionComplexCone}
Let $F \in \mathcal{S}(\R^5)$ with $\text{supp}(\hat{F}) \subseteq \mathcal{N}_{\delta}(\C \Gamma_2)$. Then the following estimate holds:
\begin{equation*}
\| F \|_{L^4(\R^5)} \lesssim_\varepsilon \delta^{-\varepsilon} \big\| \big( \sum_{\theta \in \Theta_{\delta,\C}} |F_{\theta}|^2 \big)^{\frac{1}{2}} \big\|_{L^4(\R^5)}.
\end{equation*}
\end{theorem}

In \cite{Maldague2022} and \cite{GuthMaldague2023}, High-Low-decompositions were applied in a sophisticated induc\-tion-\-on-\-dimension scheme to show sharp square function estimates for moment curves, cones over moment curves, and generalizations referred to as $m$th order Taylor cones. Whereas the present induction scheme is far less sophisticated, to the best of the author's knowledge, Theorem \ref{thm:SquareFunctionComplexCone} is the first example of a square function estimate for and moreover the first application of the High-Low-method to a cone over a two-parameter base curve.

\medskip

\emph{Outline of the paper.} In Section \ref{section:SquareFunctionsCurves} we show square function estimates for curves of finite type. Here we extend the C\'ordoba--Fefferman biorthogonality by considering second order differences. In Section \ref{section:SquareFunctionsCones} we show the square function estimates for the conical extensions. By a more local argument compared to \cite{GuthWangZhang2020}, we show a generalized Kakeya estimate, which in the degenerate case deviates logarithmically from the case of the non-degenerate cone. By dyadic pigeonholing and rescaling, we can reduce the square function estimate for the degenerate cones to the case of non-degenerate cones. In Section \ref{section:SquareFunctionComplexCone} we show Theorem \ref{thm:SquareFunctionEstimateDegenerateCone} building on the local analysis from the previous section.


\section{Square function estimates for curves of finite type}
\label{section:SquareFunctionsCurves}

In this section we prove Theorem \ref{thm:DegenerateSquareFunctionEstimateCurves}. We begin with the motivating case $k=3$, which can be carried out very explicitly.

\begin{proof}[Proof~of~Theorem~\ref{thm:DegenerateSquareFunctionEstimateCurves},~$k=3$]
By symmetry and finite decomposition we can suppose that the Fourier support of $F$ is contained in $\mathcal{N}_\delta( \{(\xi,\xi^k) : \xi \in [0,1] \})$.

By Plancherel's theorem, we find
\begin{equation}
\label{eq:PlancherelCurve}
\begin{split}
\big\| \sum_{\theta} F_{\theta} \big\|_{L^4(\R^2)} &= \int_{\R^2} \big| \big( \sum_{\theta} F_\theta \big)^2 \big|^2 \\
&= \int_{\R^2} \big| \sum_{\theta_1} \hat{F}_{\theta_1} * \sum_{\theta_2} \hat{F}_{\theta_2} \big|^2 \\
&= \sum_{\substack{\theta_1,\theta_2, \\ \theta_3, \theta_4}} \int_{\R^2} \big( \hat{F}_{\theta_1} * \hat{F}_{\theta_2} \big) \overline{\big( \hat{F}_{\theta_3} * \hat{F}_{\theta_4} \big)}.
\end{split}
\end{equation}
So, for $\theta_1,\ldots,\theta_4 \in \Theta_{\delta,3}$ making a contribution, we have solutions to the system for $\xi_i \geq 0$, $(\xi_i,\xi_i^3) \in \theta_i$:
\begin{equation}
\label{eq:CubicSFSystem}
\left\{ \begin{array}{cl}
\xi_1 + \xi_2 &= \xi_3 + \xi_4, \\
\xi_1^3 + \xi_2^3 &= \xi_3^3 + \xi_4^3 + \mathcal{O}(\delta).
\end{array} \right.
\end{equation}
We shall establish that the above can only be satisfied in case of \emph{essential biorthogonality}: Let $D > 1$ denote a fixed dilation factor. Then essential biorthogonality refers to
\begin{equation}
\label{eq:EssentialBiorthogonality}
(\theta_1 \subseteq D \cdot \theta_3 \wedge \theta_2 \subseteq D \cdot
 \theta_4) \vee (\theta_1 \subseteq D \cdot \theta_4 \wedge \theta_2 \subseteq D \cdot \theta_3).
\end{equation}
Once \eqref{eq:EssentialBiorthogonality} is verified, the claim follows from \eqref{eq:PlancherelCurve} by applying the Cauchy-Schwarz inequality.

\smallskip

We turn to the analysis of \eqref{eq:CubicSFSystem}: Taking the third power of the first line and subtracting the second line we find
\begin{equation*}
\left\{ \begin{array}{cl}
\xi_1 + \xi_2 &= \xi_3 + \xi_4, \\
3 \xi_1 \xi_2 (\xi_1 + \xi_2) &= 3 \xi_3 \xi_4 (\xi_3+\xi_4) + \mathcal{O}(\delta).
\end{array} \right.
\end{equation*}
For $\max(\xi_i) \leq \delta^{\frac{1}{3}}$ we clearly have that $(\xi_i,\xi_i^3)$ all belong to the rectangle of size $C \delta^{\frac{1}{3}} \times C \delta$ centered at the origin, which settles \eqref{eq:EssentialBiorthogonality}.

\smallskip

So, we suppose that $\max(\xi_i) \geq \delta^{\frac{1}{3}}$ and obtain by dividing through $3(\xi_1+\xi_2) = 3(\xi_3+\xi_4)$:
\begin{equation*}
\left\{ \begin{array}{cl}
\xi_1 + \xi_2 &= \xi_3 + \xi_4, \\
\xi_1 \xi_2 &= \xi_3 \xi_4 + \mathcal{O}\big( \frac{\delta}{\max_i \xi_i} \big).
\end{array} \right.
\end{equation*}
Squaring the first line and subtracting the second line multiplied by $2$, we are led to the system
\begin{equation*}
\left\{ \begin{array}{cl}
\xi_1 + \xi_2 &= \xi_3 + \xi_4, \\
\xi_1^2 + \xi_2^2 &= \xi_3^2 + \xi_4^2 + \mathcal{O} \big( \frac{\delta}{\max_i \xi_i} \big).
\end{array} \right.
\end{equation*}

Now we are in the position to apply the classical C\'ordoba--Fefferman square function estimate, which yields a finitely overlapping decomposition into rectangles of length comparable to $\frac{\delta^{\frac{1}{2}}}{(\max_i \xi_i)^{\frac{1}{2}}}$, which for any $(\xi_i,h(\xi_i)) \in \theta_i$ is smaller than the tangential length of $\theta_i$. This completes the proof in case $k=3$.
\end{proof}

\begin{remark}[Mitigating~effect~of~transversality]
The fact that for $\max_i \xi_i \gg \min \xi_i$ the estimate improves reflects that, also in the degenerate case, transversality allows for improved bilinear estimates. Recall the following bilinear version of the C\'ordoba--Fefferman square function estimate \eqref{eq:ClassicalCF} for $\text{supp}(\hat{F}_i) \subseteq \mathcal{N}_\delta((\xi,\xi^2) : 0 \leq \xi \leq 1 )$ and $\text{dist}(\text{supp}(\hat{F}_1,\hat{F}_2)) \gtrsim 1$, which yields a decomposition of the Fourier support into squares $\theta \in \Theta_{\delta \times \delta, 2}$ of size $C \delta \times C \delta$:
\begin{equation*}
\| F_1 F_2 \|_{L^2(\R^2)} \lesssim \big\| \big( \sum_{\theta_1 \in \Theta_{\delta \times \delta,2}} |F_{\theta_1}|^2 \big)^{\frac{1}{2}} \big( \sum_{\theta_2 \in \Theta_{\delta \times \delta,2}} |F_{\theta_2} |^2 \big)^{\frac{1}{2}} \big\|_{L^2(\R^2)}.
\end{equation*}
\end{remark}

We turn now to the proof of the general case. The following observation will be used repeatedly:
\begin{lemma}
\label{lem:Comparability}
Let $1 \geq \xi_a \geq \xi_b \geq 0$. If $\xi_b \geq \delta^{\frac{1}{p}}$, then
\begin{equation}
\label{eq:Comp1}
\xi_a = \xi_b + \mathcal{O}\big( \frac{\delta^{\frac{1}{2}}}{\xi_b^{\frac{p-2}{2}}} \big)
\end{equation}
implies
\begin{equation}
\label{eq:Comp2}
\xi_a = \xi_b + \mathcal{O}\big( \frac{\delta^{\frac{1}{2}}}{\xi_a^{\frac{p-2}{2}}} \big).
\end{equation}
\end{lemma}
\begin{proof}
For $\xi_b \geq \delta^{\frac{1}{p}}$ it follows $\frac{\delta^{\frac{1}{2}}}{\xi_b^{\frac{p-2}{2}}} \leq \delta^{\frac{1}{p}}$, consequently \eqref{eq:Comp1} implies $\xi_a \sim \xi_b$.
\end{proof}

\begin{proof}[Proof~of~Theorem~\ref{thm:DegenerateSquareFunctionEstimateCurves}]
By symmetry and finite decomposition we can suppose that the Fourier support of $F$ is contained in $\mathcal{N}_\delta( \{(\xi,\xi^k) : \xi \in [0,1] \})$. Let $h(\xi) = \xi^k$. By the same argument as in \eqref{eq:PlancherelCurve}, we are led to the following system for $(\xi_i,h(\xi_i)) \in \theta_i$:

\begin{equation*}
\left\{ \begin{array}{cl}
\xi_1 + \xi_2 &= \xi_3 + \xi_4, \\
h(\xi_1) + h(\xi_2) &= h(\xi_3) + h(\xi_4) + \mathcal{O}(\delta).
\end{array} \right.
\end{equation*}
We turn to the verification of the essential biorthogonality in the above case.

\smallskip

We suppose by symmetry that 
\begin{equation}
\label{eq:AssumptionFrequencyDistributionA}
\xi_1 \geq \xi_2 \text{ and } \xi_3 \geq \xi_4 \text{ and } \xi_1 \geq \xi_3.
\end{equation}
In the following we suppose that
\begin{equation}
\label{eq:AssumptionFrequencyDistribution}
\xi_1 \geq \xi_3 \geq \xi_2 \geq \xi_4
\end{equation}
because the other possibility $\xi_1 \geq \xi_2 \geq \xi_3 \geq \xi_4$ immediately gives $\xi_1 = \xi_2 = \xi_3 = \xi_4$ assuming \eqref{eq:AssumptionFrequencyDistributionA}.

\smallskip

We rewrite
\begin{equation*}
\begin{split}
&\quad h(\xi_1) - h(\xi_3) + h(\xi_2) - h(\xi_4) \\
&= (\xi_1^{k-1} + \xi_1^{k-2} \xi_3 + \ldots + \xi_1 \xi_3^{k-2} + \xi_3^{k-1}) (\xi_1 - \xi_3) + (\xi_2^{k-1} + \ldots + \xi_4^{k-1}) (\xi_2 - \xi_4).
\end{split}
\end{equation*}
Since $\xi_1 - \xi_3 = \xi_4 - \xi_2$, this can be rewritten as
\begin{equation}
\label{eq:AuxEstimateI}
(\xi_1 - \xi_3) (\xi_1^{k-1} + \xi_1^{k-2} \xi_3 + \ldots + \xi_3^{k-1} - \xi_2^{k-1} - \ldots - \xi_2 \xi_4^{k-2} - \xi_4^{k-1}) = \mathcal{O}(\delta).
\end{equation}

The second factor can be estimated by \eqref{eq:AssumptionFrequencyDistribution} as
\begin{equation}
\label{eq:AuxEstimateII}
(\xi_1^{k-1} + \ldots + \xi_3^{k-1} - \xi_2^{k-1} - \ldots \xi_4^{k-1} ) \geq k \xi_3^{k-1} - k \xi_2^{k-1} = k(\xi_3 - \xi_2) (\xi_3^{k-2} + \ldots + \xi_2^{k-2}).
\end{equation}

\textbf{Case A:} $\xi_2 \leq 5 \delta^{\frac{1}{k}}$.\\
\textbf{Case A1:} $\xi_3 \leq 10 \delta^{\frac{1}{k}}$. In this case $\xi_1 \lesssim \delta^{\frac{1}{k}}$, so all points $(\xi_i,h(\xi_i))$ belong to the rectangle $C \delta^{\frac{1}{k}} \times C \delta$ at the origin.

\noindent \textbf{Case A2:} $\xi_3 \geq 10 \delta^{\frac{1}{k}}$. We find in this case
\begin{equation*}
\eqref{eq:AuxEstimateII} \gtrsim \xi_3^{k-1}.
\end{equation*}
Consequently, from \eqref{eq:AuxEstimateI} follows
\begin{equation*}
\xi_1 = \xi_3  + \mathcal{O} \big( \frac{\delta}{\xi_3^{k-1}} \big).
\end{equation*}
For $\xi_3 \geq 10 \delta^{\frac{1}{k}}$ it follows
\begin{equation*}
\frac{\delta}{\xi_3^{k-1}} \lesssim \frac{\delta^{\frac{1}{2}}}{\xi_3^{\frac{k-2}{2}}},
\end{equation*}
which shows that
\begin{equation*}
\xi_1 = \xi_3 + \mathcal{O} \big( \frac{\delta^{\frac{1}{2}}}{\xi_3^{\frac{k-2}{2}}} \big).
\end{equation*}
Since $\xi_3 \geq \delta^{\frac{1}{k}}$, we can invoke Lemma \ref{lem:Comparability} to infer
\begin{equation*}
\xi_1 = \xi_3 + \mathcal{O} \big( \frac{\delta^{\frac{1}{2}}}{\xi_1^{\frac{k-2}{2}}} \big).
\end{equation*}
Consequently, as well $(\xi_1,h(\xi_1))$, $(\xi_3,h(\xi_3))$ as $(\xi_2,h(\xi_2))$, $(\xi_4,h(\xi_4))$ belong to essentially the same $\theta_i$ (the latter two essentially the one at the origin).

\textbf{Case B:} $\xi_2 \geq 5 \delta^{\frac{1}{k}}$.\\
\textbf{Case B1:} $|\xi_3 - \xi_2| \lesssim \big( \frac{\delta}{\xi_2^{k-2}} \big)^{\frac{1}{2}}.$
Invoking again Lemma \ref{lem:Comparability} we find that $\xi_2, \xi_3$ belong to essentially the same rectangle, and moreover
\begin{equation}
\label{eq:AuxB1}
\xi_1 = \xi_4 + \mathcal{O}((\delta/ \xi_2^{k-2})^{\frac{1}{2}}).
\end{equation}
\textbf{Case B1I:} $\xi_4 \gg \delta^{\frac{1}{k}}$. In this case we have by \eqref{eq:AuxB1}
\begin{equation*}
\xi_1 = \xi_4 + \mathcal{O} \big( \frac{\delta^{\frac{1}{2}}}{\xi_4^{\frac{k-2}{2}}} \big).
\end{equation*}
Invoking Lemma \ref{lem:Comparability} we obtain that $(\xi_1,h(\xi_1))$, $(\xi_4,h(\xi_4))$ belong essentially to the same $\theta_i$.\\
\textbf{Case B1II:} $\xi_4 \lesssim \delta^{\frac{1}{k}}$. It follows from \eqref{eq:AuxB1} and $\xi_2 \geq 5 \delta^{\frac{1}{k}}$ that $\xi_1 \lesssim \delta^{\frac{1}{k}}$. Consequently, all $(\xi_i,h(\xi_i))$ belong essentially to the same $\theta_i$ at the origin.

\smallskip

\noindent \textbf{Case B2:} $ |\xi_3 - \xi_2 | \gtrsim \big( \frac{\delta}{\xi_2^{k-2}} \big)^{\frac{1}{2}}$.

Taking \eqref{eq:AuxEstimateI} and \eqref{eq:AuxEstimateII} and the assumption \eqref{eq:AssumptionFrequencyDistribution} in this case together gives
\begin{equation*}
|\xi_1 - \xi_3 | \lesssim \delta / (\big( \frac{\delta}{\xi_2^{k-2}} \big)^{\frac{1}{2}} \cdot \xi_3^{k-2} ) \lesssim \big( \delta^{\frac{1}{2}} / \xi_3^{(k-2)/2} \big).
\end{equation*}
Since $\xi_3 \geq \xi_2 \geq 5 \delta^{\frac{1}{k}}$, we can invoke Lemma \ref{lem:Comparability}, which shows indeed that $\xi_1,\xi_3$ belong to essentially the same rectangle. Consequently, we have
\begin{equation*}
| \xi_2 - \xi_4 | \lesssim \big( \delta^{\frac{1}{2}} / \xi_3^{(k-2)/2} \big) \lesssim \big( \delta^{\frac{1}{2}} / \xi_4^{(k-2)/2} \big).
\end{equation*}
This underlines that as well $\xi_2$, $\xi_4$ belong to essentially the same rectangle. The proof is complete.

\end{proof}

\section{Square function estimates for cones over degenerate curves}
\label{section:SquareFunctionsCones}

In the following we extend the square function estimate from the previous section to cones. Let $k \in \N_{\geq 2}$. Presently, we denote the base curve by $\gamma_k = \{ (\omega_1,\omega_2) \in \R^2 : \, \omega_1 \in (-1,1), \, \omega_2 = f_k(\omega_1) \}$, which generates the truncated cone:
\begin{equation*}
\mathcal{C} \gamma_k = \{ \omega_3 \cdot (\omega_1/\omega_3,\omega_2,1) \in \R^3 : \omega_2 = f_k(\omega_1/\omega_3), \quad 0 \leq |\omega_1| \leq 1, \; \frac{1}{2} \leq \omega_3 \leq 1 \}.
\end{equation*}

By finite decomposition and rigid motion, we suppose in the following that $\omega_1 \geq 0$. We require that $f_k \in C^k(0,1) \cap C([0,1])$ with $f_k(0) = 0$ and there are $C_m \geq 1$, $m=1,\ldots,k$ such that for all $\omega \in (0,1)$:
\begin{equation}
\label{eq:DerivativeBounds}
C_m^{-1} \leq f_k^{(m)}(\omega) / \omega^{k-m} \leq C_m \text{ for } 1 \leq m \leq k.
\end{equation}
The estimates will be uniform in $\mathcal{C} \gamma_k$ upon imposing a bound $C_m \leq C_m^*$.

%

\smallskip

 Let $\delta > 0$. We parametrize the canonical covering of $\mathcal{N}_\delta(\gamma_p)$, which covers the $\delta$-neighbourhood with rectangles of maximal tangential length: An unnormalized tangential vector is given by $\mathbf{t}(\omega) = (1,f'_k(\omega))$.
An inner normal vector is given by $\mathbf{n}(\omega) = (-f'_k(\omega),1)$.

\smallskip

For $|\omega_1| \lesssim \delta^{\frac{1}{k}}$, we choose $\mathcal{O}(1)$ rectangles of length $C \delta^{\frac{1}{k}}$ into the tangential direction and length $C \delta$ into normal direction.

For $|\omega_1| \gg \delta^{\frac{1}{k}}$ we carry out a dyadic decomposition $|\omega_1| \sim K \in 2^{\Z}$ and choose points $\omega_1$ separated of length $\frac{\delta^{\frac{1}{2}}}{K^{\frac{k-2}{2}}}$. The rectangles are then chosen of length $\frac{C \delta^{\frac{1}{2}}}{K^{\frac{k-2}{2}}}$ into tangential direction and of length $C \delta$ into normal direction. Note that to cover the dyadic region of the curve with $|\xi_1| \sim K$, we require $K / (\delta^{\frac{1}{2}} / K^{\frac{k-2}{2}}) \sim K^{\frac{k}{2}} / \delta^{\frac{1}{2}}$ rectangles.

We denote a collection of centers for the rectangles obtained from this process as $\mathcal{R}(\delta) \subseteq \R^2$ and $\mathcal{R}_1(\delta) = \pi_1(\mathcal{R}(\delta))$.

Now we consider with the above parametrization the conical extension: For $\xi \in \gamma_k$\footnote{Here we abuse notation and denote the mapping and its image synonymously.}, we define the central line and normal vectors via
\begin{equation*}
\mathbf{c}(\xi) = (\xi_1,\xi_2,1), \qquad \mathbf{n}_{\gamma}(\xi)= (-f'_k(\xi_1),1,0), \quad \mathbf{n}_3 = (0,0,1).
\end{equation*}
Note that these are not normalized, but have Euclidean norm comparable to $1$. We choose as tangent vector of the base curve $(\xi_1,\xi_2,1)$, $\xi \in \gamma_k$: $\mathbf{t}(\xi) = (1,f_k'(\xi_1),0)$. Note again that $\mathbf{t}$ is unnormalized but has modulus comparable to $1$. We can cover $\mathcal{N}_\delta(\mathcal{C} \gamma_k)$ with rectangles
\begin{equation*}
\theta(\delta,\xi)= \{ a \mathbf{c}(\xi) + b \mathbf{t}(\xi) + c_1 \mathbf{n}_\gamma(\xi) + c_2 \mathbf{n}_3 : \; \frac{1}{2} \leq a \leq 1, \; |b| \leq c_0 \frac{\delta^{\frac{1}{2}}}{|\xi_1|^{\frac{k-2}{2}}}, \; |c_1|,|c_2| \leq c_0 \delta \}
\end{equation*}
for $|\xi_1| \gtrsim \delta^{\frac{1}{k}}$ and some $c_0 > 1$. For $\xi \in \gamma_k$ with $|\xi_1| \lesssim \delta^{\frac{1}{k}}$ we consider
\begin{equation*} 
\theta(\delta,\xi) = \{ a \mathbf{c}(\xi) + b \mathbf{t}(\xi) + c_1 \mathbf{n}_\gamma(\xi) + c_2 \mathbf{n}_3 : \; \frac{1}{2} \leq a \leq 1, \; |b| \leq c_0 \delta^{\frac{1}{k}}, \; |c_1|,|c_2| \leq c_0 \delta \}.
\end{equation*}
In the following $\theta^*$ denotes the polar set.

\smallskip

We define based on the canonical covering of the $\delta$-neighbourhood of $\gamma_k$:
\begin{equation*}
\Theta_\delta = \{ \theta (\delta,\xi) : \xi \in \mathcal{R}(\delta) \}.
\end{equation*}
We sort the $\theta$ according to sections of comparable curvature: For $K \in 2^{\Z}$ with $K \gtrsim \delta^{\frac{1}{k}}$ we define
\begin{equation}
\label{eq:SortingSectorsCurvature}
\Theta_\delta(K) = \{ \theta(\delta,\xi) : \xi \in \mathcal{R}(\delta), \; \xi_1 \sim K \}.
\end{equation}
Additionally, we let for the $\mathcal{O}(1)$-sectors close to the origin:
\begin{equation*}
\Theta_\delta(0) = \{ \theta(\delta,\xi) : \xi \in \mathcal{R}(\delta), \; |\xi_1| \lesssim \delta^{\frac{1}{k}} \}.
\end{equation*}

We show the following:
\begin{theorem}
\label{thm:SquareFunctionEstimateDegenerateCone}
Let $0 < \delta \ll 1$ and $F \in \mathcal{S}(\R^3)$ with $\text{supp}(\hat{F}) \subseteq \mathcal{N}_{\delta}(\mathcal{C} \gamma_k)$. Then the following estimate holds:
\begin{equation*}
\| F \|_{L^4(\R^3)} \lesssim_\varepsilon \delta^{-\varepsilon} \big\| \big( \sum_{\theta \in \Theta_\delta} |F_\theta|^2 \big)^{\frac{1}{2}} \big\|_{L^4(\R^3)}.
\end{equation*}
\end{theorem}

For the proof we shall extend the High-Low-method pioneered in \cite{GuthWangZhang2020}. This establishes a Kakeya-type estimate for the overlap of the planks $\tilde{\theta} = \theta(\xi) - \theta(\xi)$, which is the Fourier support of $|f_\theta|^2$.

\subsection{A generalized Kakeya estimate}

We prove the following estimate, which generalizes \cite[Lemma~1.4]{GuthWangZhang2020}:
\begin{proposition}
\label{prop:GeneralizedKakeya}
Let $r \gg 1$, $\delta = r^{-2}$, and $k \geq 2$. Suppose that $\text{supp}(\hat{f}) \subseteq \mathcal{N}_\delta(\mathcal{C} \gamma_k)$. Then the following estimate holds:
\begin{equation}
\label{eq:GeneralizedKakeya}
\int_{\R^3} \big( \sum_{\theta \in \Theta_\delta} |f_{\theta}|^2 \big)^2 \lesssim \log(r^{-1}) \sum_{r^{-1} \leq s \leq 1} \sum_{\tau \in \Theta_{s^2}} \sum_{U \parallel U_{\tau,r^2}} |U|^{-1} \| S_U f \|^4_{L^2(U)}
\end{equation}
with $U_{\tau,r^{-2}} = \text{conv}( \bigcup_{\substack{\theta \in \Theta_{r^{-2}}, \\ \theta \subseteq \tau}} \theta^* )$.
\end{proposition}


\smallskip

\begin{remark}
For $k=2$, the analysis yields that the estimate holds without logarithmic loss; see Remark \ref{rem:DifferentScales} (2).
\end{remark}
%

Like in the case of the circular cone, we shall obtain incidence estimates for dyadic heights $2^{\Z} \ni h = \sigma^2$, $r^{-1} \leq \sigma \leq 1$. In the following, to simplify notations, we let $\delta = r^{-2}$. We define centered planks at scale $\sigma$: these are taylored such that at height $h = \sigma^2$ the centered planks cover canonically the $r^{-2}$-neighbourhood of the degenerate curve. To this end, we rescale the small height to unit height: This inflates the $r^{-2}$-neighbourhood to the $r^{-2} \sigma^{-2}$-neighbourhood of the degenerate cone at unit length. Now we choose $\xi \in \mathcal{R}(r^{-2} \sigma^{-2})$, i.e., the spacing associated with the canonical covering of the neighbourhood of size $r^{-2} \sigma^{-2}$.

But the length of the rectangles has to be rescaled again by $\sigma^2$. This leads to the following: For $\xi \in \mathcal{R}(r^{-2} \sigma^{-2})$ with $|\xi_1| \lesssim (r \sigma)^{-\frac{2}{k}}$ consider
\begin{equation*}
\begin{split}
\Theta(\sigma,\xi) &= 
\{ a \mathbf{c}(\xi) + b \mathbf{t}(\xi) + c_1 \mathbf{n}_\gamma(\xi) + c_2 \mathbf{n}_3 : \; |a| \leq \sigma^2, \; |b| \leq C (r \sigma)^{-\frac{2}{k}} \sigma^2, \\
&\quad |c_1|,|c_2| \leq C r^{-2} \}.
\end{split}
\end{equation*}
For $\xi \in \mathcal{R}(r^{-2} \sigma^{-2})$ with $|\xi_1| \gg (r \sigma)^{-\frac{2}{k}}$, $|\xi_1| \sim K \in 2^{\Z}$, $K \ll 1$ consider
\begin{equation*}
\Theta(\sigma,\xi) = \{ 
a \mathbf{c}(\xi) + b \mathbf{t}(\xi) + c_1 \mathbf{n}_\gamma(\xi) + c_2 \mathbf{n}_3 : \; |a| \leq \sigma^2, \; |b| \leq C \frac{r^{-1} \sigma}{K^{\frac{k-2}{2}}}, \; |c_1|,|c_2| \leq C r^{-2} \}.
\end{equation*}
$C$ will be chosen $C=C(c_0,C_m)$ with $C_m$ given by \eqref{eq:DerivativeBounds}.

These centered planks at scale $\sigma$ form a mild dilation of the canonical covering of the $r^{-2}$-neighbourhood of $h \gamma_4$. We denote the collection by
\begin{equation*}
\mathbf{CP}_{\sigma} = \{ \Theta(\sigma,\xi) : \xi \in \mathcal{R}(r^{-2} \sigma^{-2}) \}.
\end{equation*}
For a set $A$ we denote the collection of its subelements by
\begin{equation*}
\bigcup A = \{ x \; | \; \exists y \in A:  x \in y  \}.
\end{equation*}

\smallskip

In the High-Low-decomposition we will consider differences
\begin{equation}
\label{eq:DifferenceDecomposition}
\bigcup \mathbf{CP}_{\sigma} \backslash \bigcup \mathbf{CP}_{\sigma/2}.
\end{equation}
The union will cover the Fourier support of $ \sum_{\theta \in \Theta_\delta} |f_{\theta}|^2 $ and on each union we shall obtain a certain almost orthogonality decomposition into coarser planks.

\smallskip

Define for $r^{-1} \leq \sigma \leq 1$, $\sigma \in 2^{\Z}$:
\begin{equation*}
\Omega_{\leq \sigma} = \bigcup \mathbf{CP}_{\sigma} \text{ and } \Omega_{\sigma} = \Omega_{\leq \sigma} \backslash \Omega_{\leq \sigma/2}.
\end{equation*}
Note that $\bigcup \tilde{\theta} \subseteq \bigcup_{\xi \in \mathcal{R}(r^{-2})} \Theta(1,\xi)$ (the $\tilde{\theta}$ correspond to $\Theta(1,\xi)$ up to a mild dilation).

Let $\xi,\xi' \in \gamma_k$. We associate $\theta(\xi')$ to $\Theta(\sigma,\xi)$ as follows:
For $|\xi_1'| \lesssim (r \sigma)^{-\frac{2}{k}}$ choose $\xi$ with $|\xi_1 - \xi_1'| \leq (r \sigma)^{-\frac{2}{k}}$. For $|\xi_1'| \sim K \gg (r \sigma)^{-\frac{2}{k}}$, choose $\xi$ with $|\xi_1-\xi_1'| \leq \frac{r^{-1} \sigma^{-1}}{K^{\frac{k-2}{2}}}$. We write $\theta(\xi') \in \Theta(\sigma,\xi)$.

\begin{proof}[Proof~of~Proposition~\ref{prop:GeneralizedKakeya}]

With the above sorting of sectors \eqref{eq:SortingSectorsCurvature}, we write
\begin{equation*}
\sum_{\theta \in \Theta_\delta} |f_\theta|^2 \lesssim \sum_{K \in [\delta^{\frac{1}{k}},1] \cup \{ 0 \} } \sum_{\theta \in \Theta_\delta(K)} |f_\theta|^2.
\end{equation*}
By dyadic pigeonholing we obtain for some $K \in [\delta^{\frac{1}{k}},1] \cup \{ 0 \}$:
\begin{equation}
\label{eq:CurvaturePigeonholing}
\sum_{\theta \in \Theta_\delta} |f_\theta|^2 \lesssim \log(r^{-1}) \sum_{\theta \in \Theta_\delta(K)} |f_\theta|^2.
\end{equation}
In the case $K \lesssim r^{-\frac{2}{k}}$ there are only $\mathcal{O}(1)$-sectors and the conclusion of the argument follows from a simple variant of the arguments below. In the following we suppose that $K \gg r^{-\frac{2}{k}}$.

Let $g(x) = \sum_{\theta \in \Theta_\delta(K)} \big( |f_{\theta}|^2 \big)(x)$. Following the pigeonholing, we can refine the definition of $\mathbf{CP}_{\sigma}$. Let
\begin{equation*}
\mathbf{CP}'_{\sigma} = \{ \Theta(\sigma,\xi) : \, \xi_1 \sim K \}
\end{equation*}
and correspondingly, for $\sigma \sim \frac{r^{-1}}{K^{\frac{k-2}{2}}}$ we have that $\# \{ \Theta(\sigma,\xi) : \Theta(\sigma,\xi) \in \mathbf{CP}'_{\sigma} \} \lesssim 1$. This suggests to carry out the decomposition into $\Omega_\sigma$ up to a dyadic scale $\sigma_0 \gg \frac{r^{-1}}{K^{\frac{k-2}{2}}}$. We define for
$\sigma_0 \leq \sigma \leq 1$, $\sigma \in 2^{\Z}$:
\begin{equation*}
\Omega'_{\leq \sigma} = \bigcup \mathbf{CP}_{\sigma}' \text{ and } \Omega'_{\sigma} = \Omega'_{\leq \sigma} \backslash \Omega'_{\leq \sigma/2}.
\end{equation*}

We shall prove a High-Low-estimate after invoking Plancherel's theorem:
\begin{equation*}
\int_{\R^3} |g(x)|^2 = \int_{\R^3} |\hat{g}(\omega)|^2 = \int_{\R^3} \big| \sum_{\theta \in \Theta_\delta(K)} \big( |f_\theta|^2 \big) \widehat{\,} (\omega) \big|^2 = \int_{\Omega_{\leq 1}} \big| \sum_{\theta \in \Theta_\delta(K)} \big( |f_\theta|^2 \big) \widehat{\,} (\omega) \big|^2.
\end{equation*}
The pigeonholing carried out in \eqref{eq:CurvaturePigeonholing} will be implicit in the following to lighten the notation. Moreover, for $\sigma \sim \frac{r^{-1}}{K^{\frac{k-2}{2}}}$ we have that $\# \{ \Theta(\sigma,\xi) : \Theta(\sigma,\xi) \in \mathbf{CP}_{\sigma} \} \lesssim 1$. This suggests to carry out the decomposition \eqref{eq:DifferenceDecomposition} up to a dyadic scale $\sigma_0 \gg \frac{r^{-1}}{K^{\frac{k-2}{2}}}$. We abuse notation and denote $\Omega'_{\leq \sigma_0}$ by $\Omega'_\sigma$ again.

\smallskip

We consider the dyadic partition:
\begin{equation}
\label{eq:DyadicPartitionHighLow}
\int_{\Omega'_{\leq 1}} \big| \big( \sum_{\theta \in \Theta_\delta(K)} |f_{\theta}|^2(\omega) \big|^2 = \sum_{\sigma_0 \leq \sigma \leq 1 } \underbrace{\int_{\Omega'_\sigma} \big| \sum_{\theta \in \Theta_\delta(K)} \big( |f_\theta|^2 \big) \widehat{\,} (\omega) \big|^2 d\omega}_{A(\sigma)}.
\end{equation}
We sort the sum over $\theta$ into $\Theta(\sigma,\xi)$:
\begin{equation*}
A(\sigma)= \int_{\Omega'_{\sigma}} \big| \sum_{\Theta(\sigma,\xi) \in \mathbf{CP}'_\sigma} \sum_{\substack{\theta \in \Theta(\sigma,\xi), \\ \theta \in \Theta_{\delta}(K)}} \big( |f_{\theta}|^2 \big) \widehat{\,} (\omega) \big|^2.
\end{equation*}
We shall see, extending the crucial observation from the non-degenerate case, that for $\omega \in \Omega_\sigma'$, the overlap of $\Theta(\sigma,\xi)$ is finite. This allows us to apply the Cauchy-Schwarz inequality without significant loss:
\begin{equation}
\label{eq:FiniteOverlap}
\int_{\Omega_\sigma'} \big| \sum_{\Theta(\sigma,\xi) \in \mathbf{CP}'_\sigma} \sum_{\theta \in \Theta(\sigma,\xi)} \big( |f_\theta|^2 \big) \widehat{\,} (\omega) \big|^2 \lesssim \int_{\Omega'_\sigma} \sum_{\Theta(\sigma,\xi) \in \mathbf{CP}'_\sigma} \big| \sum_{\theta \in \Theta(\sigma,\xi)} \big( |f_{\theta}|^2 \big) \widehat{\,} (\omega) \big|^2.
\end{equation}

For $\sigma \sim \sigma_0$, we have $|\mathbf{CP}'_{\sigma}| \sim 1$ by $|\xi| \sim K$. Consequently, \eqref{eq:FiniteOverlap} is immediate for $\sigma \sim \sigma_0$. In the following we suppose that $\sigma \gg \sigma_0$.

\medskip

It will be crucial to understand the set $\Omega'_{\sigma} \cap \{\omega_3 = h \}$. Note that for $\theta=\theta(\xi)$ with $\xi \in \mathcal{R}(r^{-2})$ we obtain for $\tilde{\theta}(\xi) \cap \{ \omega_3 = h \}$ a rectangle centered at $h \gamma_4(\xi_1)$ with length $\frac{r^{-1}}{K^{\frac{k-2}{2}}}$ into the tangential direction and $r^{-2}$ into the normal direction:
\begin{equation*}
\begin{split}
&\quad \pi_{12} (\text{supp}(\mathcal{F} (|f_\theta|^2)) \cap \{ \omega_3 = h \}) \\
&= \{ h \gamma_4(\xi_1) + \ell \dot{\gamma}_4(\xi_1) + c r^{-2} \mathbf{n}_{\gamma}(\xi_1) : |\ell| \lesssim \frac{r^{-1}}{K^{\frac{k-2}{2}}}, \; |c| \leq 2 \}.
\end{split}
\end{equation*}
By $\pi_{12}: \R^3 \to \R^2$ we denote the projection to the first two coordinates, which will often be implicit in the following when intersecting with $\{ \omega_3 = h\}$. In the following to unify notation in the above display, we let $\mathbf{n}_{\gamma}(\xi_1) = \mathbf{n}_\gamma(\xi)$. By symmetry $\omega_3 \to - \omega_3$ we restrict in the following to non-negative $h$.

\medskip

 To show \eqref{eq:FiniteOverlap}, we begin with the following lemma, which states that for $h \ll \sigma^2$, if a point lies in the tangential neighbourhood of $h \gamma_k(\eta_1)$ with a tangential distance much smaller than $\ell(\Theta(\sigma,\eta))$, then it lies in $\mathbf{CP}_{\sigma/2}$.
 
For this we do not have to require $|\eta_1| \sim K$ or a minimum size condition $\sigma \gg \sigma_0$.
\begin{lemma}
\label{lem:ComparabilityTheta}
\begin{enumerate}
Let $\sigma \in 2^{\Z} \cap [r^{-1},1]$, $0 \leq h \leq \sigma^2$, and $|\eta_1| \ll 1$.
\item Let 
\begin{equation}
\label{eq:RepPoint(A)}
p = h \gamma_k(\eta_1) + \ell \dot{\gamma}_k(\eta_1) + c r^{-2} \mathbf{n}_{\gamma}(\eta_1)
\end{equation}
 with $|\eta_1| \ll (r \sigma)^{-\frac{2}{k}}$, $|\ell| \ll (r \sigma)^{-\frac{2}{k}} \sigma^2$, and $|c| \leq c_0$. Then we have
\begin{equation}
\label{eq:RepPoint(B)}
p = h \gamma_k(0) + \ell_1 \dot{\gamma}_k(0) + C r^{-2} \mathbf{n}_{\gamma}(0)
\end{equation}
with $|C| \leq C^*(c_0,C_m)$ with $C_m$ given in \eqref{eq:DerivativeBounds} and $|\ell_1| \ll (r \sigma)^{-\frac{2}{k}} \sigma^2$.
Furthermore, for $p$ defined by \eqref{eq:RepPoint(A)} with $|\eta_1| \ll (r \sigma)^{-\frac{2}{k}}$, $|\ell| \sim (r \sigma)^{-\frac{2}{k}} \sigma^2$, we obtain \eqref{eq:RepPoint(B)} with $|\ell_1| \sim (r \sigma)^{-\frac{2}{k}} \sigma^2$.
\item
Let $p$ be defined by \eqref{eq:RepPoint(A)} with $|\eta_1| \gtrsim (r \sigma)^{-\frac{2}{k}}	$, $|\ell| \ll \frac{r^{-1} \sigma}{|\eta_1|^{\frac{k-2}{2}}}$, and $|c| \leq c_0$. Then there is $\xi \in \mathcal{R}(r^{-2} (\sigma/2)^{-2})$ with $|\xi_1 - \eta_1| \lesssim \frac{r^{-1} \sigma^{-1}}{|\eta_1|^{\frac{k-2}{2}}}$ and
\begin{equation}
\label{eq:Representation(1)}
p = h \gamma_4(\xi_1) + \ell_1 \dot{\gamma}_4(\xi_1) + C r^{-2} \mathbf{n}_{\gamma}(\xi_1)
\end{equation}
with $|C| \leq C^*(c_0,C_m)$ and $|\ell_1| \ll r^{-1} \sigma/|\xi_1|^{\frac{k-2}{2}}$, which implies for $h \ll \sigma^2$:
\begin{equation*}
p \in \Theta(\sigma/2,\xi) \cap \{ \omega_3 = h \}.
\end{equation*}
Furthermore, for $p$ defined by \eqref{eq:RepPoint(A)} with $|\eta_1| \lesssim (r \sigma)^{-\frac{2}{k}}$, $|\ell| \sim \frac{r^{-1} \sigma}{|\eta_1|^{\frac{k-2}{2}}}$, and $|c| \leq 2$ we have \eqref{eq:Representation(1)} with $|C| \leq C^*$ and $|\ell_1| \sim \frac{r^{-1} \sigma}{|\xi_1|^{\frac{k-2}{2}}}$.


\end{enumerate}
\end{lemma}
\begin{proof}
We can suppose that $\sigma \gg r^{-1}$ since for $\sigma \sim r^{-1}$ there are only finitely many $\xi \in \mathcal{R}(r^{-2} \sigma^{-2})$.

\medskip

\emph{Proof of (1):} We compute under the above assumptions by Taylor expansion:
\begin{equation*}
\begin{split}
&\quad h \gamma_4(\eta_1) + \ell \dot{\gamma}_4(\eta_1) + c r^{-2} \mathbf{n}_{\gamma}(\eta_1) \\
&= h (\gamma_k(0) + \frac{\eta_1^k}{k!} \gamma^{(k)}_k(\bar{\xi}_a)) + \ell (1,f_k'(\eta_1))  + c r^{-2} \mathbf{n}_{\gamma}(\eta_1) 
\end{split}
\end{equation*}
and recall that $|f_k'(\eta_1)| \sim |\eta_1|^{k-1}$, $|f_k^{(k)}| \sim 1$.

Consequently, for $|\eta_1| \ll (r \sigma)^{-\frac{2}{k}}$, $|\ell| \ll (r \sigma)^{-\frac{2}{k}} \sigma^2$, and $|c| \leq 2$, we find
\begin{equation*}
h \gamma_4(\eta_1) + \ell \dot{\gamma}_4(\eta_1) + c r^{-2} \mathbf{n}_{\gamma}(\eta_1) = (h \eta_1 + \ell + c r^{-2} \mathbf{n}_{\gamma,1}(\eta_1), h + c_1 r^{-2})
\end{equation*}
with $|c_1| \leq |c_0|+1$. This yields the representation
\begin{equation*}
h \gamma_4(\eta_1) + \ell \dot{\gamma}_4(\eta_1) + c r^{-2} \mathbf{n}_{\gamma}(\eta_1) = h(0,1) + \ell_1 (1,0) + C r^{-2}(0,1)
\end{equation*}
with $|\ell_1| \ll (r \sigma)^{-\frac{2}{k}} \sigma^2$ and $|C| \leq C^*(c_0,C_m)$.

\smallskip

The same computation shows that for $|\eta_1| \lesssim (r \sigma)^{-\frac{2}{k}}$ and $|\ell_1| \sim (r \sigma)^{-\frac{2}{k}} \sigma^2$ we obtain a representation as
\begin{equation*}
h \gamma_k(0) + \ell \dot{\gamma}_k(0) + C r^{-2} \mathbf{n}_{\gamma}(0)
\end{equation*}
with $|\ell| \sim (r \sigma)^{-\frac{2}{k}} \sigma^2$.
On the other hand, for $|\eta_1| \lesssim (r \sigma)^{-\frac{2}{k}}$ and $|\ell_1| \gg (r \sigma)^{-\frac{2}{k}} \sigma^2$ we have that $p \notin \Theta(\sigma,0) \cap \{ \omega_3 = h \}$.

\medskip

\emph{Proof of (2):} Let $p = h \gamma_k(\eta_1) + \ell \dot{\gamma}_k(\eta_1) + c r^{-2} \mathbf{n}_{\gamma}(\eta_1)$ with $|\eta_1| \gtrsim (r \sigma)^{-\frac{2}{k}}$, $|\ell| \ll \frac{r^{-1} \sigma}{|\eta_1|^{\frac{k-2}{2}}}$, and $|c| \leq c_0$. To show the representation \eqref{eq:Representation(1)}, we carry out a Taylor expansion at $\xi_1 \sim \eta_1$ with $|\Delta \xi_1| \sim \frac{r^{-1} \sigma}{|\xi_1|^{\frac{k-2}{2}}}$ to find
\begin{equation*}
\begin{split}
&\quad h \gamma_k(\xi_1 + \Delta \xi_1) + \ell \dot{\gamma}_k(\xi_1 + \Delta \xi_1) + c r^{-2} \mathbf{n}_{\gamma}(\xi_1 + \Delta \xi_1) \\
&= h (\gamma_k(\xi_1) + (\Delta \xi_1) \dot{\gamma}_k(\xi_1) + \mathcal{O}(h (\Delta \xi_1)^2 \ddot{\gamma}_k(\xi_1)) + \\
&\quad \ell (\dot{\gamma}_k(\xi_1) + \mathcal{O}( (\Delta \xi_1) \ddot{\gamma}_k(\xi_1)) + c r^{-2} \mathbf{n}_{\gamma}(\xi_1).
\end{split}
\end{equation*}
Recalling that $|\ddot{\gamma}_k(\xi_1)| \sim |\xi_1|^{k-2}$ such that by hypothesis we have
\begin{equation*}
|h (\Delta \xi_1)^2 \ddot{\gamma}_k(\xi_1)| \ll r^{-2} \text{ and } |\ell (\Delta \xi_1) \ddot{\gamma}_k(\xi_1)| \ll r^{-2}.
\end{equation*}
This allows us to write
\begin{equation*}
h \gamma_k(\eta_1) + (\ell+h (\Delta \eta_1)) \dot{\gamma}_k(\eta_1)+ C r^{-2} \mathbf{e},
\end{equation*}
with $|\mathbf{e}| \leq 1$ and $|C| \leq C^*$. $\mathbf{e}$ can then be decomposed into tangential and normal vector, which yields the desired representation.

\end{proof}

In the following we rely on $\sigma \gg \sigma_0$ and $|\eta_1| \sim K$. We shall see that the points $p \in \tilde{\theta}(\eta) \cap \Omega'_\sigma \cap \{ \omega_3 = h \}$ for $h \ll \sigma^2$ with the above representation
\begin{equation*}
p = h \gamma_k(\eta_1) + \ell \dot{\gamma}_k(\eta_1) + c r^{-2} \mathbf{n}_{\gamma}(\eta)
\end{equation*}
actually satisfy $|\ell| \sim \frac{r^{-1} \sigma}{K^{\frac{k-2}{2}}}$. Together with the previous lemma, we have the following representation:
\begin{lemma}
\label{lem:TwoEndsRepresentation}
Let $|\eta_1| \sim K$, $h \leq \sigma^2$, $p \in \tilde{\theta}(\eta_1) \cap \{ \omega_3 = h \} \cap \Omega'_\sigma$ and $\sigma \gg \sigma_0 \sim \frac{r^{-1}}{K^{\frac{k-2}{2}}}$.
%
We have the representation
\begin{equation}
\label{eq:TwoEndsRepresenation2}
p = h \gamma_k(\xi_1) + \ell \dot{\gamma}_k(\xi_1) + C r^{-2} \mathbf{n}_{\gamma}(\xi_1)
\end{equation}
with $\xi_1 \in \mathcal{R}_1(r^{-2} \sigma^{-2})$, $|\xi_1 - \eta_1| \leq 2 \frac{r^{-1} \sigma^{-1}}{|\xi_1|^{\frac{k-2}{2}}}$, $|\ell| \lesssim \frac{r^{-1} \sigma}{K^{\frac{k-2}{2}}}$, and $|C| \leq C^*(c_0,C_m)$. For $h \ll \sigma^2$ \eqref{eq:TwoEndsRepresenation2} holds with $|\ell| \sim \frac{r^{-1} \sigma}{K^{\frac{k-2}{2}}}$.
\end{lemma}

We remark that as a consequence of the assumptions, it follows that $|\eta_1| \gtrsim (r \sigma)^{-\frac{2}{k}}$ and that $\xi_1$ is always bounded away from zero.

A consequence of the lemma is that $\Omega'_\sigma$ at height $h \ll \sigma^2$ for $\sigma \gg \sigma_0$ consists entirely of a mild dilation of the ends of $\Theta(\sigma,\xi) \cap \{ \omega_3 = h \}$.

\smallskip

Under the above assumptions on $|\xi|$ and $\sigma \gg \sigma_0$, we define the right end of $\Theta(\sigma,\xi) \cap \{ \omega_3 = h \}$ as collection of points:
\begin{equation*}
\text{RE}(\Theta(\sigma,\xi),h) = 
\{ h \gamma_k(\xi_1) + \ell \dot{\gamma}_k(\xi_1) + C r^{-2} \mathbf{n}_{\gamma}(\xi_1) \, : \, \ell \sim \frac{r^{-1} \sigma}{|\xi_1|^{\frac{k-2}{2}}}, \quad |C| \leq C^*(c_0,C_m) \}.
\end{equation*}
The left end of $\Theta(\sigma,\xi) \cap \{ \omega_3 = h \}$ is correspondingly defined as
\begin{equation*}
\text{LE}(\Theta(\sigma,\xi),h) = 
\{ h \gamma_k(\xi_1) + \ell \dot{\gamma}_4(\xi_1) + C r^{-2} \mathbf{n}_{\gamma}(\xi_1) \, : \, \quad \ell \sim - \frac{r^{-1} \sigma}{|\xi_1|^{\frac{k-2}{2}}}, \quad |C| \leq C^* \}.
\end{equation*}

We turn to the proof of Lemma \ref{lem:TwoEndsRepresentation}:
\begin{proof}[Proof~of~Lemma~\ref{lem:TwoEndsRepresentation}]

We shall see that a point
\begin{equation*}
p = h \gamma_k(\eta_1) + \ell \dot{\gamma}_k(\eta_1) + c r^{-2} \mathbf{n}_{\gamma}(\eta)
\end{equation*}
with $|\ell| \gg \frac{r^{-1} \sigma}{K^{\frac{k-2}{2}}}$ cannot be covered with any $\Theta(\sigma,\xi)$. For $h \sim \sigma^2$ recall that
\begin{equation*}
\mathcal{N}_{c_1 r^{-2}}(h \gamma_k)  \subseteq \bigcup \mathbf{CP}'_{\sigma} \cap \{ \omega_3 = h \} \subseteq \mathcal{N}_{c_2 r^{-2}}(h \gamma_k),
\end{equation*}
but $p = h \gamma_k(\eta_1) + \ell \dot{\gamma}_k(\eta_1) \notin \mathcal{N}_{c_2 r^{-2}}(h \gamma_k)$ for $|\ell| \gg \frac{r^{-1} \sigma}{K^{\frac{k-2}{2}}}$
as a consequence of Taylor's theorem. 

\smallskip

In the following we turn to $h \ll \sigma^2$ and suppose that there exists a point 
\begin{equation*}
p = h \gamma_k(\eta_1) + \ell \dot{\gamma}_k(\eta_1) + c r^{-2} \mathbf{n}_{\gamma}(\eta_1) \in \Omega_\sigma' \cap \{ \omega_3 = h \}
\end{equation*}
with $|\ell| \gg \frac{r^{-1} \sigma}{K^{\frac{k-2}{2}}}$. We shall see that the right side of $\tilde{\theta}(\eta)$ at this length cannot be covered by the right nor the left side of $\Theta(\sigma,\xi)$.

\smallskip

\emph{Case 1:} $\ell \gg \frac{r^{-1} \sigma}{K^{\frac{k-2}{2}}}$. We shall see that the ``long" right side can neither be covered by the right nor the left side of $\Theta(\sigma,\xi)$. 

\emph{Case 1a:} We exclude the case that $p$ is covered by the right side of some $\Theta(\sigma,\xi)$, for which we argue by contradiction.
Suppose that for $\ell_1 \geq 0$:
\begin{equation}
\label{eq:Rep(P)}
h \gamma_k(\eta_1) + \ell \dot{\gamma}_k(\eta_1) + c r^{-2} \mathbf{n}_{\gamma}(\eta_1) = h \gamma_k(\xi_1) + \ell_1 \dot{\gamma}_k(\xi_1) + C r^{-2} \mathbf{n}_{\gamma}(\xi_1).
\end{equation}
We can suppose that $\ell_1 \gg r^{-2}$ since otherwise, $p \in \mathcal{N}_{C r^{-2}}(h \gamma_k)$ but $\mathcal{N}_{C r^{-2}}(h \gamma_4)$ is covered by $\mathbf{CP}'_{\sigma/2}$, in which case $p \notin \Omega'_\sigma$.

Since $\ell_1 \lesssim \frac{r^{-1} \sigma}{K^{\frac{k-2}{2}}}$, $h \ll \sigma^2$, by projection to the first coordinate we have $|\xi_1 - \eta_1| \gg \frac{r^{-1} \sigma^{-1}}{K^{\frac{k-2}{2}}}$. We shall see that necessarily the right ends of $\Theta(\sigma,\xi)$ are essentially disjoint. 
\begin{figure}[ht!]
\label{fig:DisjointEnds}
\centering
\includegraphics[width=100mm]{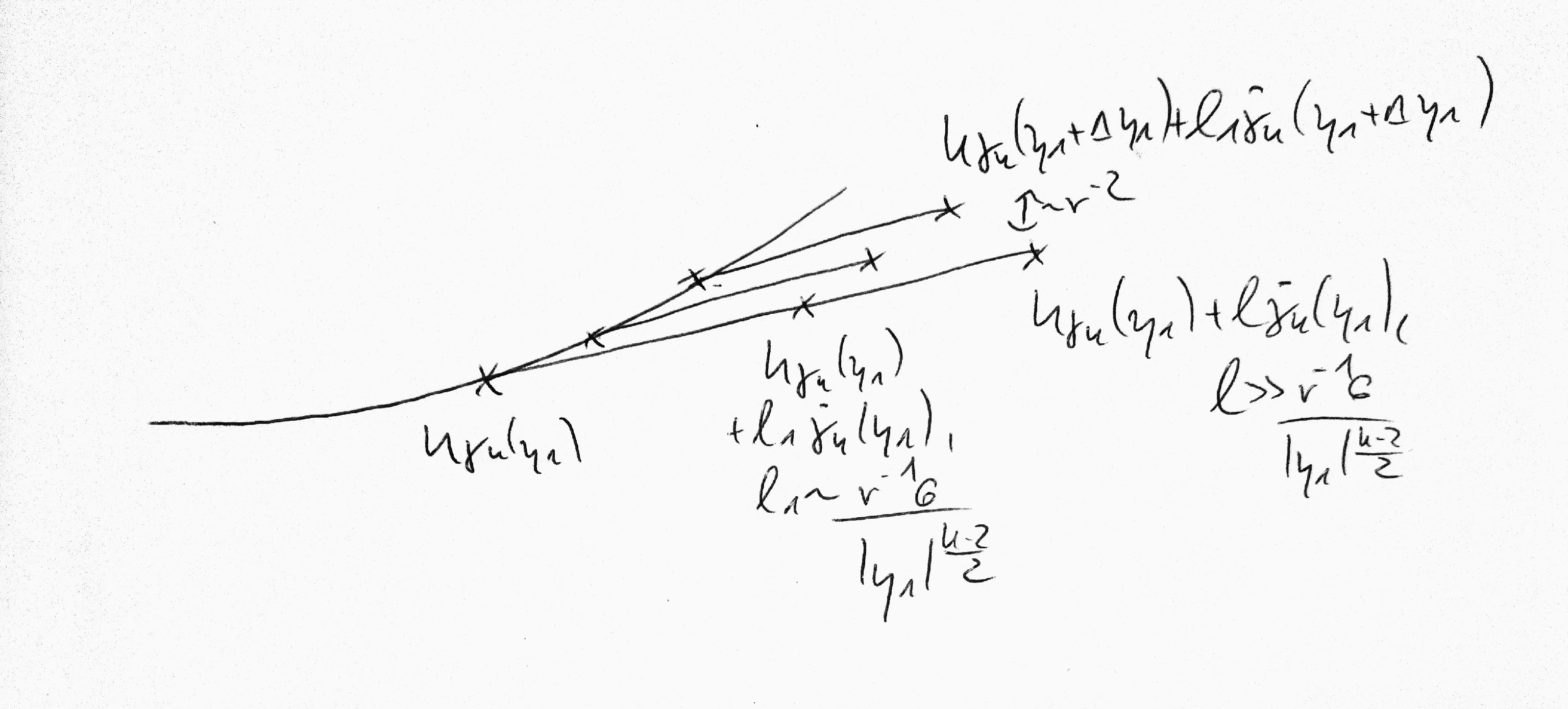}
\caption{Essential disjointness of the ends of $\Theta(\sigma,\xi)$.}
\end{figure}

Let $\Delta \xi_1= \bar{C} \frac{r^{-1} \sigma^{-1}}{K^{\frac{k-2}{2}}}$ for $\bar{C} \gg 1$ to find by Taylor expansion with Lagrange remainder
\begin{equation}
\label{eq:TaylorRightEnds}
\begin{split}
&\quad h \gamma_k(\xi_1 + \Delta \xi_1) + \ell \dot{\gamma}_k(\xi_1+\Delta \xi_1) + C r^{-2} \mathbf{n}_{\gamma}(\xi_1 + \Delta \xi_1) \\
&= h (\gamma_k(\xi_1) + (\Delta \xi_1) \dot{\gamma}_k(\xi_1) + \frac{ (\Delta \xi)^2}{2} \ddot{\gamma}_k(\bar{\xi}_a) ) \\
&\quad + \ell (\dot{\gamma}_k(\xi_1) + \Delta \xi_1 \ddot{\gamma}_k(\xi_1) + \frac{(\Delta \xi)^2}{2} \gamma_k^{(3)}(\bar{\xi}_b)) + C r^{-2} \mathbf{n}_{\gamma}(\xi_1 + \Delta \xi_1).
\end{split}
\end{equation}

We estimate by the derivative bounds for $f_k$:
\begin{equation*}
\big| \frac{h (\Delta \xi)^2}{2} \ddot{\gamma}_k(\bar{\xi}_a) \big| + \big| \frac{\ell (\Delta \xi_1)^2}{2} \gamma_k^{(3)}(\bar{\xi}_b) \big| \ll r^{-2}.
\end{equation*}

Moreover, 
\begin{equation*}
\langle \mathbf{e}_2, \ell (\Delta \xi_1) \ddot{\gamma}_k(\xi_1) \rangle \sim \langle \mathbf{n}_{\gamma}(\eta_1), \ell (\Delta \xi_1) \ddot{\gamma}_k(\xi_1) \rangle \gg r^{-2}.
\end{equation*}
This shows that the right end of $\Theta(\sigma,\xi)$ has too much distance to $h \gamma_k(\eta_1) + \ell \dot{\gamma}_k(\eta_1)$ in the direction of $\mathbf{n}_{\gamma}(\eta_1)$ for $\xi_1 = \eta_1 + \Delta \xi_1$ to intersect with $\tilde{\theta}(\eta)$.

\medskip

\emph{Case 1b:} Next, we shall exclude the case that $p$ can be covered by the left side of $\Theta(\sigma,\xi)$. Suppose that for $p \in \Omega_{\sigma} \cap \{ \omega_3 = h \}$:
\begin{equation}
\label{eq:Case1bOverlap}
h \gamma_k(\eta_1) + \ell \dot{\gamma}_k(\eta_1) + c r^{-2} \mathbf{n}_{\gamma}(\eta_1) = p = h \gamma_k(\xi_1) + \ell_1 \dot{\gamma}_k(\xi_1) + C r^{-2} \mathbf{n}_{\gamma}(\xi_1).
\end{equation}
\begin{figure}[ht!]
\label{fig:LeftEndRightEnd}
\centering
\includegraphics[width=90mm]{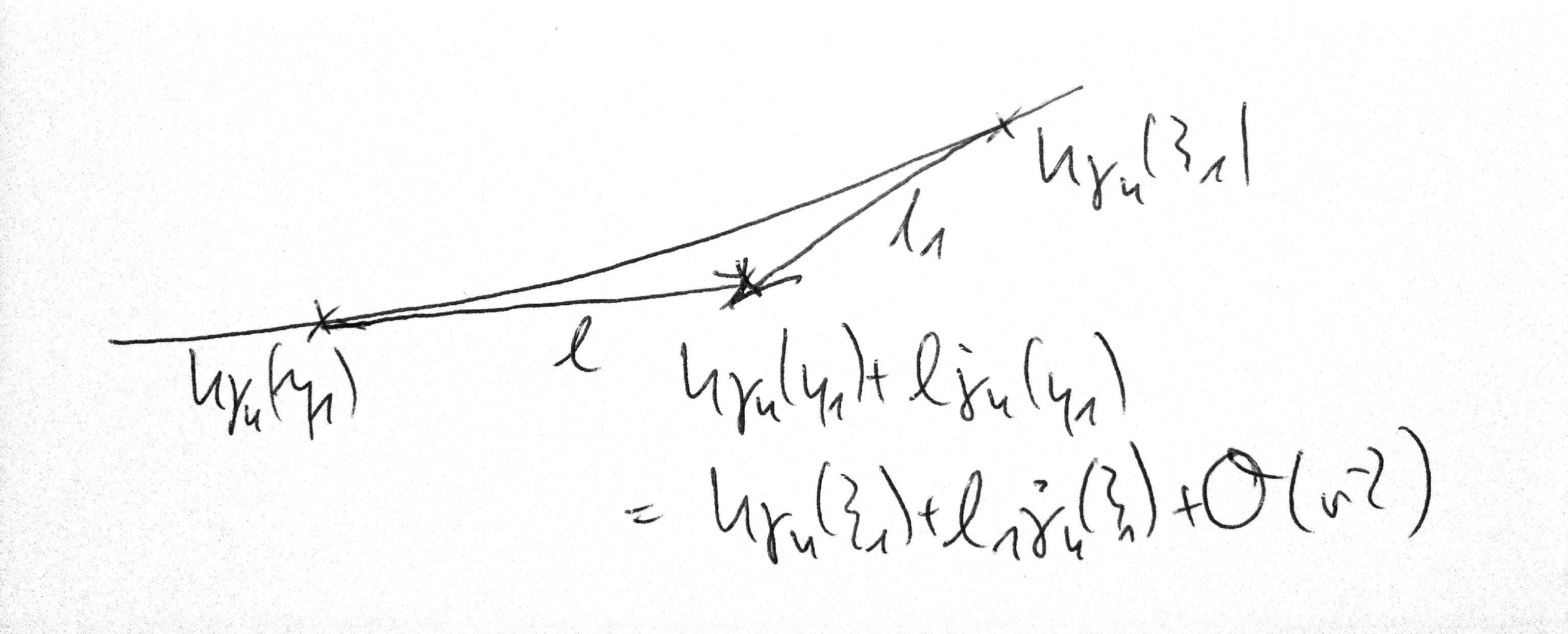}
\caption{We shall see that in case the long end of $\tilde{\theta}(\eta)$ touches the other end of $\Theta(\sigma,\xi)$ the distance to the curve is $\mathcal{O}(r^{-2})$.}
\end{figure}

Suppose that $\ell \gg \frac{r^{-1} \sigma}{|\eta_1|^{\frac{k-2}{2}}} \gtrsim |\ell_1|$. We carry out a Taylor expansion of the right hand side of \eqref{eq:Case1bOverlap} to find
\begin{equation*}
\begin{split}
&\quad h \gamma_k(\xi_1) + \ell_1 \dot{\gamma}_k(\xi_1) \\
&= h ( \gamma_k(\eta_1 + \Delta \xi_1)) + \ell_1 \dot{\gamma}_k(\eta_1 + \Delta \xi_1) \\
&= h ( \gamma_k(\eta_1) + \Delta \xi_1 \dot{\gamma}_k(\eta_1) + \frac{(\Delta \xi_1)^2}{2} \ddot{\gamma}_k(\eta_1) + \ldots + \frac{(\Delta \xi_1)^k}{k!} \gamma_k^{(k)}(\bar{\xi}_a)) \\ 
&\quad + \ell_1 ( \dot{\gamma}_k(\eta_1) + \Delta \xi_1 \ddot{\gamma}_k(\eta_1) + \ldots + \frac{(\Delta \xi_1)^{k-1}}{(k-1)!} \gamma_k^{(k)}(\bar{\xi}_b) ).
\end{split}
\end{equation*}

Projecting to the first coordinate yields
\begin{equation}
\label{eq:LengthComparabilityI}
\ell = h \Delta \xi_1 + \ell_1 + \mathcal{O}(r^{-2}) \text{ and consequently, } \ell \sim h \Delta \xi_1 \gg |\ell_1|.
\end{equation}
Projecting to the second coordinate yields
\begin{equation}
\label{eq:LengthComparabilityII}
\begin{split}
&\quad \big( \frac{h \Delta \xi_1}{2} + \ell_1 \big) \Delta \xi_1 \ddot{\gamma}_k(\eta_1) + \big( \frac{h \Delta \xi_1}{3!} + \frac{\ell_1}{2} \big) (\Delta \xi_1)^2 \gamma_k^{(3)}(\eta_1) + \ldots \\
&+ \big( \frac{h \Delta \xi_1}{k!} \gamma_k^{(k)}(\bar{\xi}_a) + \frac{\ell_1}{(k-1)!} \gamma_k^{(k)}(\bar{\xi}_b) \big) ( \Delta \xi_1)^{k-1} = \mathcal{O}(r^{-2}).
\end{split}
\end{equation}

By further finite subdivision of the dyadic range $|\eta_1| \sim |\xi_1| \sim K$ we can arrange that $|\Delta \xi_1| \ll K$. Together with the derivative bounds of $f_k$ we obtain from plugging \eqref{eq:LengthComparabilityI} into \eqref{eq:LengthComparabilityII} that
\begin{equation*}
h (\Delta \xi_1)^2 \ddot{\gamma}_k(\eta_1) \lesssim r^{-2}.
\end{equation*}
Now we obtain from another Taylor expansion that
\begin{equation*}
\begin{split}
&\quad h \gamma_k(\eta_1) + h (\Delta \xi_1) \dot{\gamma}_k(\eta_1) + c r^{-2} \mathbf{n}_{\gamma}(\eta_1) \\
&= h \gamma_k(\eta_1 + \Delta \eta_1) + \frac{h (\Delta \xi_1)^2}{2} \ddot{\gamma}_k(\bar{\xi}_c) + cr^{-2} \mathbf{n}_{\gamma}(\eta_1) \\
&= h \gamma_k(\eta_1 + \Delta \xi_1) + \mathcal{O}(r^{-2}).
\end{split}
\end{equation*}

By convexity, this gives for the left hand side of \eqref{eq:Case1bOverlap} for $\ell \leq h \Delta \xi_1$:
\begin{equation*}
\text{lhs} \eqref{eq:Case1bOverlap} = h \gamma_k(\eta_1) + \ell \dot{\gamma}_k(\eta_1) + c r^{-2} \mathbf{n}_{\gamma}(\eta_1) \in \mathcal{N}_{C r^{-2}}(h \gamma_k),
\end{equation*}
but recalling that $\mathcal{N}_{C r^{-2}}(h \gamma_k) \subseteq \bigcup \mathbf{CP}_{\sigma/2}$, this shows that $p \notin \Omega_\sigma$.

\medskip

\emph{Case 2:} The case of  $p \in \tilde{\theta}(\eta) \cap \Omega_\sigma$ being located on the long left side of $\tilde{\theta}$ can be ruled out by mirroring the above arguments.

\end{proof}

\begin{remark}
\label{rem:DifferentScales}
\begin{enumerate}
\item The Taylor expansion in \eqref{eq:TaylorRightEnds} shows the essential disjointness of the right ends of $\Theta(\sigma,\xi)$ and likewise for the left ends, \emph{mutatis mutandis}.
\smallskip

\item It is for this lemma we require the dyadic pigeonholing $|\xi_1| \sim K$. If we do not impose this constraint, there can be an unfavorable additional overlap between different scales:
\begin{figure}[ht!]
\label{fig:CurvatureDeviation}
\centering
\includegraphics[width=90mm]{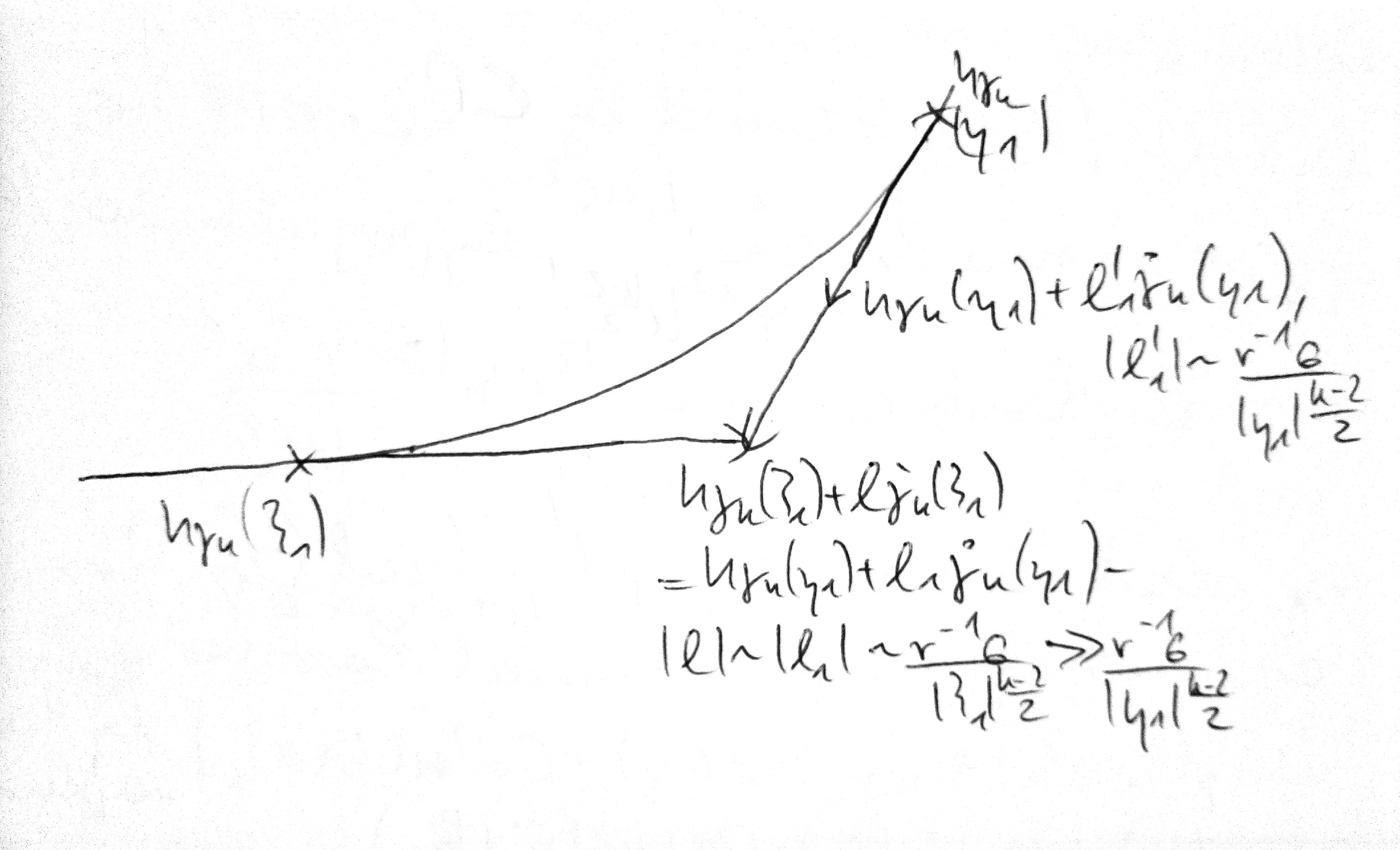}
\caption{Non-trivial interaction between regions with significantly different curvature.}
\end{figure}
It becomes possible that a long left side of $\tilde{\theta}(\eta_1)$ at a scale $\eta_1 \gg \xi_1$ touches the right end of $\Theta(\sigma,\xi_1)$
\begin{equation*}
h \gamma_k(\xi_1) + \ell_1 \dot{\gamma}_k(\xi_1)  = p = h \gamma_k(\eta_1) + \ell \dot{\gamma}_k(\eta_1)
\end{equation*}
 with 
\begin{equation*}
|\ell| \sim \ell_1 \sim \frac{r^{-1} \sigma}{|\xi_1|^{\frac{k-2}{2}}} \gg \frac{r^{-1} \sigma}{|\eta_1|^{\frac{k-2}{2}}}.
\end{equation*}
This can happen because the ``canonical" scale of $\Theta(\sigma,\eta_1)$ is much smaller than the one of $\Theta(\sigma,\xi_1)$ due to a significant change of the curvature. It is conceivable that this interaction between different scales corresponds to a logarithmic loss compared to the case of the circular cone. Indeed, the base curve of the circular cone has constant curvature, for which reason the pigeonholing into regions with comparable curvature is not necessary for the above argument. This moreover points out that in case $k=2$, the estimate \eqref{eq:GeneralizedKakeya} holds without logarithmic loss.
\end{enumerate}
\end{remark}

The following is immediate from the inclusion property $\Theta(\sigma,\xi) \subseteq M \Theta(\sigma,\xi')$ for neighbouring base points $\xi$, $\xi'$.

\begin{corollary}
\label{cor:StructureOmega1}
Let $|\eta| \sim K$, $\sigma \geq \sigma_0$, and $\omega \in \tilde{\theta}(\eta) \cap \Omega'_\sigma$.

Then, we have $\omega \in S \Theta(\sigma,\xi)$ for $\theta(\eta) \in \Theta(\sigma,\xi)$. $S$ depends on $C^*$ and $M$.
\end{corollary}

Secondly,
\begin{corollary}
\label{cor:StructureOmega2}
Let $\omega \in \Omega'_\sigma$ with $\omega_3 = h \ll \sigma^2$ and $\sigma \gg \sigma_0$. Then 
\begin{equation*}
\pi_{12}(\omega) \in S \text{Ends}(h) = S \big( \bigcup_{\xi \in \mathcal{R}(r^{-2} \sigma^{-2})} \text{LE}(\Theta(\sigma,\xi),h) \cup \bigcup_{\xi \in \mathcal{R}(r^{-2} \sigma^{-2})} \text{RE}(\Theta(\sigma,\xi),h)).
\end{equation*}
\end{corollary}

We complete the first step of the proof of \eqref{eq:FiniteOverlap} invoking Corollary \ref{cor:StructureOmega1}:
\begin{equation}
\label{eq:ProofFiniteOverlap1}
\int_{\Omega'_\sigma} \big| \sum_{\Theta(\sigma,\xi) \in \mathbf{CP}'_{\sigma}} \big| \sum_{\substack{\theta \in \Theta(\sigma,\xi), \\ \xi \sim K}} \big(|f_{\theta}|^2\big) \widehat{\,}(\omega) \big| \big|^2 \lesssim \int_{\Omega'_\sigma} \big| \sum_{\substack{\Theta(\sigma,\xi) \in \mathbf{CP}'_{\sigma}, \\ \omega \in S \Theta(\sigma,\xi)}} \big| \sum_{\substack{\theta \in \Theta(\sigma,\xi), \\ \xi \sim K}} \big(|f_{\theta}|^2\big) \widehat{\,}(\omega) \big| \big|^2.
\end{equation}

In the second step we shall see that there are only finitely many $\Theta(\sigma,\xi)$ which are contributing:

\begin{lemma}
\label{lem:AuxOverlapII}
Let $\omega \in \Omega'_\sigma$. Then the following estimate holds:
\begin{equation*}
\# \{ \Theta(\sigma,\xi) : \omega \in S \Theta(\sigma,\xi) \} \lesssim 1.
\end{equation*}
\end{lemma}
\begin{proof}
Note again that we can suppose $\sigma \gg \sigma_0$, since for $\sigma \sim \sigma_0$ it holds 
\begin{equation*}
\# \{ \Theta(\sigma,\xi) \in \mathbf{CP}'_{\sigma} \} \lesssim 1.
\end{equation*}

\smallskip

Firstly, suppose $h \sim \sigma^2$. Rescaling to unit height, we find that $h^{-1} ( \Theta(\sigma,\xi) \cap \{ \omega_3 = h \} )$ forms a canonical covering of $\gamma_k$ at scale $r^{-2} \sigma^{-2}$. Since the canonical covering is finitely overlapping, this settles the case $h \sim \sigma^2$.

\smallskip

Next, we turn to $h \ll \sigma^2$. We are in the position to invoke Corollary \ref{cor:StructureOmega2}. Then the finite overlap is a consequence of the finite overlap of the right ends and left ends separately. This follows from the Taylor expansions \eqref{eq:TaylorRightEnds} already carried out in Lemma \ref{lem:TwoEndsRepresentation}.

\end{proof}
We can now prove \eqref{eq:FiniteOverlap}.
We obtain for $\omega \in \Omega_\sigma'$:
\begin{equation*}
\begin{split}
\big| \sum_{\Theta(\sigma,\xi) \in \mathbf{CP}'_\sigma} \sum_{\substack{\theta \in \Theta(\sigma,\xi), \\ \xi \sim K}} \big( |f_\theta|^2 \big) \widehat{\,} (\omega) \big| &\lesssim \big| \sum_{\substack{\Theta(\sigma,\xi) \in \mathbf{CP}'_\sigma, \\ \omega \in S \Theta(\sigma,\xi)}} \sum_{\theta \in \Theta(\sigma,\xi)} \big( |f_{\theta}|^2 \big) \widehat{\,} (\omega) \big| \\
&\lesssim \big( \sum_{\substack{ \Theta(\sigma,\xi) \in \mathbf{CP}'_\sigma, \\ \omega \in S \Theta(\sigma,\xi)}} \big| \sum_\theta \big( |f_{\theta}|^2 \big) \widehat{\,}(\omega) \big|^2 \big)^{\frac{1}{2}}.
\end{split}
\end{equation*}
So integrating the above over $\Omega_\sigma'$ we obtain
\begin{equation*}
\begin{split}
\int_{\Omega'_\sigma} \big| \sum_{\Theta(\sigma,\xi) \in \mathbf{CP}'_\sigma} \sum_{\substack{\theta \in \Theta(\sigma,\xi), \\ \xi \sim K}} \big( |f_{\theta}|^2 \big) \widehat{\,} (\omega) \big|^2 d\omega &\lesssim \int_{\Omega_\sigma'} \sum_{\substack{ \Theta(\sigma,\xi) \in \mathbf{CP}'_\sigma, \\ \omega \in S \Theta(\sigma,\xi) }} \big| \sum_{\theta \in \Theta(\sigma,\xi)} \big( |f_\theta|^2 \big) \widehat{\,} (\omega) \big|^2 \\
&\lesssim \int_{\R^3} \sum_{\Theta(\sigma,\xi) \in \mathbf{CP}_\sigma} \big| \sum_{\substack{\theta \in \Theta(\sigma,\xi), \\ \xi \sim K}} \big( |f_\theta|^2 \big) \widehat{\,}(\omega) \big|^2 \\
&\lesssim \sum_{\Theta(\sigma,\xi) \in \mathbf{CP}_\sigma} \int_{\R^3} \big| \sum_{\theta \in \Theta(\sigma,\xi)} |f_{\theta}|^2 \big|^2.
\end{split}
\end{equation*}
The final estimate is again due to Plancherel's theorem. We can omit the localization to scales $\xi \sim K$ in the following.

\medskip

We let $\Theta(\sigma,\xi) \in \mathbf{CP}_\sigma$ bijectively correspond to $\tau \in \Theta_{r^{-2} \sigma^{-2}}$ via the base points and we can dominate
\begin{equation*}
\sum_{\theta \in \Theta(\sigma,\xi)} |f_{\theta}|^2 \leq \sum_{\theta \subseteq S \tau} |f_{\theta}|^2.
\end{equation*}

 $\theta \in \Theta(\sigma,\xi)$ corresponds to $\theta \subseteq S \tau$. Consequently, we can estimate
\begin{equation*}
\sum_{\Theta(\sigma,\xi) \in \mathbf{CP}_\sigma} \int_{\R^3} \big| \sum_{\theta \in \Theta(\sigma,\xi)} |f_{\theta}|^2 \big|^2 \lesssim \sum_{\tau \in \Theta_{r^{-2} \sigma^{-2}}} \int_{\R^3} \big| \sum_{\substack{\theta \in \Theta_{r^{-2}}: \\ \theta \subseteq S \tau}} |f_\theta|^2 \big|^2.
\end{equation*}

We have proved after pigeonholing \eqref{eq:CurvaturePigeonholing}, carrying out the sum over $\sigma \in [\sigma_0,1]$, and changing notation $(r \sigma)^{-1} = s$ that
\begin{equation}
\label{eq:AuxKakeyaEstimate}
\begin{split}
\int_{\R^3} \big| \sum_{\theta \in \Theta_{r^{-2}}} |f_{\theta}|^2 \big|^2 dx &\lesssim \log(r^{-1})
\sum_{r^{-1} \leq \sigma \leq 1} \sum_{\tau \in \Theta_{r^{-2} \sigma^{-2}}} \int_{\R^3} \big| \sum_{\substack{\theta \in \Theta_{r^{-2}}: \\ \theta \subseteq S \tau}} |f_{\theta}|^2 \big|^2 \\
&\lesssim \log(r^{-1}) \sum_{r^{-1} \leq s \leq 1} \sum_{\tau \in \Theta_{s^2}} \int_{\R^3} \big( \sum_{\substack{\theta \in \Theta_{r^{-2}}: \\ \theta \subseteq S \tau}} |f_{\theta}|^2 \big)^2.
\end{split}
\end{equation}
It remains to divide up the final integral into the dual regions of the Fourier support, which is a consequence of the uncertainty principle. 

\smallskip

To this end, choose a smooth function $\eta_{\tau}$, which is identical to one on $S \tilde{\tau}$ and rapidly decaying away from $S \tilde{\tau}$. Moreover, we require that the inverse Fourier transform $\check{\eta}_{\tau}$ is supported on $C \tilde{\tau}^* \approx U_{\tau,r^2}$ and satisfies $|\check{\eta}_{\tau}(x)| \lesssim | \tau^* |^{-1}$. By Plancherel's theorem and Fourier inversion we can break the integral into translates of $U_{\tau,r^2}$ to find
\begin{equation*}
\int_{\R^3} \big| \check{\eta}_{\tau} * \big( \sum_{\substack{ \theta \in \Theta_{r^{-2}}, \\ \theta \subseteq S \tau}} |f_{\theta}|^2 \big) \big|^2 = \sum_{U \parallel U_{\tau,r^2}} \int_U \big| \check{\eta}_{\tau} * \big( \sum_{\substack{ \theta \in \Theta_{r^{-2}}, \\ \theta \subseteq S \tau}} |f_{\theta}|^2 \big) \big|^2 dx.
\end{equation*}
We have for each $x \in U$:
\begin{equation*}
\big| \check{\eta}_{\tau} * \sum_{\substack{ \theta \in \Theta_{r^{-2}}, \\ \theta \subseteq S \tau}} |f_{\theta}|^2 (x) \big| \lesssim |U|^{-1} \int \eta_U \sum_{\substack{ \theta \in \Theta_{r^{-2}}, \\ \theta \subseteq S \tau}} |f_{\theta}|^2 (x) dx,
\end{equation*}
where $\eta_U$ denotes an $L^\infty$-normalized bump function supported in $CU$.

We obtain
\begin{equation}
\label{eq:PartitionEstimate}
\begin{split}
\sum_{U \parallel U_{\tau,r^2}} \int_U \big| \check{\eta}_{\tau} * \big( \sum_{\substack{ \theta \in \Theta_{r^{-2}}, \\ \theta \subseteq S \tau}} |f_{\theta}|^2 \big) \big|^2 dx &\lesssim \sum_{U \parallel U_{\tau,r^2}} |U|^{-1} \big( \eta_U \sum_{\substack{ \theta \in \Theta_{r^{-2}}, \\ \theta \subseteq S \tau}} |f_{\theta}|^2 (x) \big)^2 \\
&\lesssim \sum_{U \parallel U_{\tau,r^2}} |U|^{-1} \| S_U f \|^4_{L^2(U)}.
\end{split}
\end{equation}
The ultimate estimate follows from the support of $\eta_U$ being comparable to $U$. Plugging \eqref{eq:PartitionEstimate} into \eqref{eq:AuxKakeyaEstimate}, we find
\begin{equation*}
\int_{\R^3} \big( \sum_{\theta \in \Theta_{r^{-2}}} |f_{\theta}|^2 \big)^2 \lesssim \log(r^{-1}) \sum_{r^{-1} \leq s \leq 1} \sum_{\tau \in \Theta_{s^2}} \sum_{U \parallel U_{\tau,r^2}} |U|^{-1} \| S_U f \|^4_{L^2(U)}.
\end{equation*}
The proof of \eqref{eq:GeneralizedKakeya} is complete.

\end{proof}

\subsection{A slicing argument}

In the following we prove the conical square function estimate for the degenerate cone:
\begin{equation*}
\mathcal{C} \gamma_k = \{ (\xi_1,\xi_1^k / \xi_3^{k-1},\xi_3 ) : |\xi_1| \leq 1, \quad \xi_3 \in [1/2,1] \}.
\end{equation*}
$\Theta_\delta$ denotes the canonical covering of the $\delta$-neighbourhood introduced previously.
\smallskip

Instead of following the roadmap from Guth--Wang--Zhang \cite{GuthWangZhang2020}, we rely on pigeonholing, rescaling, and the stability result for Kakeya estimates for non-\-dege\-nerate cones discussed above. This incurs a logarithmic loss compared to the square function estimate for degenerate curves, but already for the circular cone it is currently not clear whether the loss can be removed.

\smallskip

This corresponds to the observation recorded in Remark \ref{rem:DifferentScales} that we do not expect the generalized Kakeya estimate to hold globally for cones over finite-type curves.
We have the following stability result of \cite[Theorem~1.1]{GuthWangZhang2020}:
\begin{theorem}
\label{thm:StabilitySquareFunction}
Let $f_2 \in C^2(-1,1)$ with
\begin{equation}
\label{eq:Regularity}
f_2'(0) = 0 \text{ and } \exists c_2, C_2: \forall \xi \in (0,1): \, c_2 \leq f_2''(\xi) \leq C_2. 
\end{equation}
Let $\mathcal{C} \gamma_2 = \{ \omega_3 (\omega_1/\omega_3,f_2(\omega_1/\omega_3),1) : \, \omega_1 \in [-1,1], \; \omega_3 \in [1/2,1] \}$ denote the cone with base curve $\gamma_2 = \{(\xi, f_2(\xi) \}$. Let $f \in \mathcal{S}(\R^3)$ with $\text{supp}(\hat{f}) \subseteq \mathcal{N}_{R^{-1}}(\mathcal{C} \gamma_2)$. Then the following estimate holds:
\begin{equation*}
\| f \|_{L^4(\R^3)} \lesssim_{\varepsilon,c_2,C_2} R^\varepsilon \big\| \big( \sum_{\theta \in \Theta_{R^{-1}}} |f_{\theta}|^2 \big)^{\frac{1}{2}} \big\|_{L^4(\R^3)}.
\end{equation*}
\end{theorem}

\smallskip

Now we establish the square function estimate for the degenerate cone $\mathcal{C} \gamma_k$, $k \geq 3$. 

\begin{proof}[Proof~of~Theorem~\ref{thm:SquareFunctionEstimateDegenerateCone}]

To apply Theorem \ref{thm:StabilitySquareFunction}, we carry out a dyadic pigeonholing into sectors. Let $M \in 2^{\Z} \cap [R^{-\frac{1}{k}},1]$, and let
\begin{equation*}
\nu_M = \{ (\xi_1,\xi_2,\xi_3) \in \mathcal{C} \gamma_k : \big| \frac{\xi_1}{\xi_3} - \nu \big| \lesssim M \}.
\end{equation*}
Additionally, we let
\begin{equation*}
\nu_0 = \{ (\xi_1,\xi_2,\xi_3) \in \mathcal{C} \gamma_k : \big| \frac{\xi_1}{\xi_3} \big| \lesssim R^{-\frac{1}{k}} \},
\end{equation*}
which contains $\mathcal{O}(1)$-sectors $\theta \in \Theta_{R^{-1}}$ close to the origin.

By dyadic pigeonholing, it suffices to establish a square function estimate for one of the dyadic regions defined above. Let $M \in (2^{\Z} \cap [R^{-\frac{1}{k}},1]) \cup \{ 0 \}$ such that
\begin{equation}
\label{eq:PigeonholingSquareFunction}
\| f \|_{L^4(\R^3)} \lesssim \log(R) \| f_M \|_{L^4(\R^3)}
\end{equation}
with $f_M$ denoting the Fourier projection of $f$ to the $R^{-1}$-neighbourhood of $\nu_M$.

Clearly, for $M \lesssim R^{-\frac{1}{k}}$, since only finitely many $\theta \in \Theta_{R^{-1}}$ are contributing, it is immediate from the Cauchy-Schwarz inequality that
\begin{equation}
\label{eq:SFEstimateSmallM}
\| f_M \|_{L^4(\R^3)} \lesssim \big\| \big( \sum_{\theta \in \Theta_{R^{-1}}} |f_{\theta}|^2 \big)^{\frac{1}{2}} \big\|_{L^4(\R^3)}.
\end{equation}

\smallskip

We turn to the case $M \gg R^{-\frac{1}{k}}$: The key tool will be a generalized Lorentz rescaling, mapping $\nu_M$ to a full non-degenerate cone $\Gamma_2'$. By finite subdivision we can suppose that the coordinates for $\nu_M$ satisfy for $M = c \nu$ with $c \ll 1$ to be chosen later:
\begin{equation}
\label{eq:SectorNu}
\big| \frac{\xi_1}{\xi_3} - \nu \big| \leq M.
\end{equation}

\begin{lemma}
Let $M \in [R^{-\frac{1}{k}},1] \cap 2^{\Z}$. There is a linear map $\Lambda_\nu : \nu_M \to \Gamma_2'$, which maps the sector $\nu_M$ to a non-degenerate cone
\begin{equation*}
\Gamma_2' = \{ \omega_3 \cdot (\omega_1/\omega_3,f_2(\omega_1/\omega_3),1) : |\omega_1| \leq 1, \, \omega_3 \in [1/2,1] \}
\end{equation*}
with base curve given by
\begin{equation*}
f_2(\omega_1) = \omega_1^2 + \sum_{\ell=3}^k d_{k,\ell} c^{\ell-2} \omega_1^{\ell},
\end{equation*}
which satisfies \eqref{eq:DerivativeBounds} uniform in $M$.
$\Lambda_{\nu}$ maps the $R^{-1}$-neighbourhood of $\mathcal{C} \gamma_k$ to the $M^{-k} R^{-1}$-neighbourhood of $\Gamma_2'$.

 Secondly, the sectors at scale $\delta$ contained in $\tau$, which satisfy
\begin{equation*}
\big| \frac{\xi_1}{\xi_3} - \nu_{\theta} \big| \leq \big( \frac{\delta}{M^{k-2}} \big)^{\frac{1}{2}}
\end{equation*}
are mapped to $\delta/M^k$-sectors of $\Gamma_2'$:
\begin{equation*}
\big| \frac{\xi_1'}{\xi_3'} - \nu_{\theta'} \big| \leq \big( \frac{\delta}{M^k} \big)^{\frac{1}{2}}.
\end{equation*}
\end{lemma}
\begin{proof}
We define $\xi_1' = M^{-1}(\xi_1 - \nu \xi_3)$, $\xi_3' = \xi_3$, which transforms \eqref{eq:SectorNu} to
\begin{equation*}
\big| \frac{\xi_1'}{\xi_3'} \big| \leq 1.
\end{equation*}
We compute the effect for the degenerate cone:
\begin{equation*}
\tau' = \frac{(M \xi_1' + \nu \xi_3)^k}{\xi_3^{k-1}} = \sum_{\ell=0}^k {k \choose l} \frac{M^l \nu^{k-\ell} (\xi_1')^{\ell}}{\xi_3^{\ell-1}}.
\end{equation*}
By a linear transformation $\tau' \to \tau' + \text{Lin}(\xi_1',\xi_3')$, we find
\begin{equation*}
\tau' \to \tau' = \sum_{\ell=2}^k {k \choose l} \frac{M^l \nu^{k-\ell} (\xi_1')^{\ell}}{\xi_3^{\ell-1}}.
\end{equation*}
We rescale now with $\big( {k \choose 2} M^2 \nu^{k-2} \big)^{-1} \sim M^{-k}$ to find
\begin{equation*}
\tau' = \frac{(\xi_1')^2}{\xi_3} + \sum_{\ell=3}^k \underbrace{\big[ {k \choose \ell} / {k \choose 2} \big]}_{d_{k,\ell}} c^{\ell - 2} \frac{(\xi_1')^{\ell}}{\xi_3^{\ell-1}}.
\end{equation*}
We have obtained a linear transformation $\Lambda_\nu : \nu_M \to \Gamma_2'$, $(\xi_1,\tau,\xi_3) \mapsto (\xi_1',\tau',\xi_3')$ which maps the sector to a non-degenerate cone:
\begin{equation*}
\Lambda_\nu(\nu_M) = \Gamma_2' = \{ \omega_3 \cdot (\omega_1 / \omega_3, f_2(\omega_1 / \omega_3), 1) : \, |\omega_1| < 1, \; \omega_3 \in [1/2,1] \}
\end{equation*}
with base curve given by
\begin{equation*}
f_2(\omega_1) = \omega_1^2 + \sum_{\ell=3}^k d_{k,\ell} c^{\ell-2} \omega_1^{\ell}.
\end{equation*}
Choosing $c=c(k) \ll 1$ we find that \eqref{eq:Regularity} holds. Moreover, $\Lambda_\nu$ maps the $R^{-1}$-neighbourhood of $\mathcal{C} \gamma_k$ to the $M^{-k} R^{-1}$-neighbourhood of $\Gamma_2'$. Verifying the correspondence of the $\delta$-sectors $\nu_{\theta}$ contained in $\nu_M$ to the $\delta/M^k$-sectors of $\Gamma_2'$ is straight-forward.
\end{proof}


Let $h_M$ denote the function obtained from pulling back the Fourier transform of $f$, which satisfies $\text{supp}(\hat{h}_M) \subseteq \mathcal{N}_{C M^{-k} R^{-1}}(\Gamma_2')$ and let $J(M)$ denote the Jacobian from the change of variables. We can apply Theorem \ref{thm:StabilitySquareFunction} to find
\begin{equation}
\label{eq:RescaledSquareFunction}
\begin{split}
\| f_M \|_{L^4(\R^3)} &= J(M) \| h_M \|_{L^4(\R^3)} \\
&\lesssim_\varepsilon R^\varepsilon J(M) \big\| \big( \sum_{\theta \in \Theta_{2,M^{-k} R^{-1}}} |h_{M,\theta}|^2 \big)^{\frac{1}{2}} \big\|_{L^4(\R^3)} \\
&\lesssim_\varepsilon R^\varepsilon \big\| \big( \sum_{\theta \in \Theta_{k,R^{-1}}} |f_{\theta}|^2 \big)^{\frac{1}{2}} \big\|_{L^4(\R^3)}.
\end{split}
\end{equation}

Taking \eqref{eq:PigeonholingSquareFunction}, \eqref{eq:SFEstimateSmallM}, and \eqref{eq:RescaledSquareFunction} together, we find
\begin{equation*}
\| f \|_{L^4(\R^3)} \lesssim_\varepsilon \log(R) R^\varepsilon \big\| \big( \sum_{\theta \in \Theta_{R^{-1}} } |f_{\theta}|^2 \big)^{\frac{1}{2}} \big\|_{L^4(\R^3)},
\end{equation*}
which completes the proof.

\end{proof}

\section{A square function estimate for the complex cone}
\label{section:SquareFunctionComplexCone}

The argument presented above to show the Kakeya estimate allows for generalization to the complex cone:
\begin{equation*}
\C \Gamma_2 = \{ h \cdot (z/h, (z/h)^2 , 1 ) \in \C \times \C \times \R : h \in [1/2,1], \; z \in \C: \, |z| \leq 1 \}.
\end{equation*}
By the canonical identification $\C \equiv \R^2$ this can be regarded as cone in $\R^5$:
\begin{equation*}
\C \Gamma_2 = \{ h \cdot (s/h,t/h,(s^2-t^2)/h^2, 2st / h^2 , 1 ) \in \R^5 :\, h \in [1/2,1], \, |(s,t)| \leq 1 \}.
\end{equation*}
In the following we will frequently identify complex numbers with elements of $\R^2$ by their real and imaginary part.

\subsection{Set-up}

We denote the base curve by
\begin{equation*}
\gamma_{2,\C}:[-1,1] \times [-1,1] \ni (s,t) \mapsto (s,t,s^2-t^2,2st) \in \R^4.
\end{equation*}
More concisely, we can express this as $\gamma_{2,\C}(z) =(z,z^2) \in \C^2$ for $z \in \C \equiv \R^2$. For $0 < \delta \ll 1$ we introduce the canonical covering of $\mathcal{N}_{\delta}(\gamma_{2,\C})$. Let $z=(s,t) \in \delta^{\frac{1}{2}} (\N_0 \times \N_0) \cap [0,1]^2$. We define
\begin{equation*}
\begin{split}
\theta_z &= \{ (s,t,s^2-t^2,2st) + \ell_1 (1,0,2s,2t) + \ell_2 (0,1,-2t,2s) \\
&\quad + c_1 r^{-2} (-2s,2t,1,0) + c_2 r^{-2} (-2t,-2s,0,1) \, : \\
&\quad \quad \ell_i \in [-d \delta^{\frac{1}{2}}, d \delta^{\frac{1}{2}}], \; c_i \in [-d \delta, d \delta ] \}
\end{split}
\end{equation*}
for some $d > 1$.

Let $\Theta_{\delta} = \{ \theta_z : z = (s,t) \in \delta^{\frac{1}{2}} \Z^2 \cap [-1,1]^2 \}$.
Clearly, the $\theta_z$ form a finitely overlapping cover of $\mathcal{N}_\delta(\gamma_{2,\C})$. In complex notation this can be expressed concisely as
\begin{equation*}
\theta_z = \{ \gamma_{2,\C}(z) + \ell \dot{\gamma}_{2,\C}(z) + c \wedge (\dot{\gamma}_{2,\C}(z)): \, \ell, c \in \C, \, |\ell| \leq d \delta^{\frac{1}{2}}, \; | c | \leq d \delta \} \subseteq \C^2.
\end{equation*}
We denote with $\wedge (a,b) = -(\overline{b},\overline{a})$ for $(a,b) \in \C^2$ the vector spanning the ``complex" orthogonal complement.

The following is a consequence of complexification of the ``real" C\'ordoba--Fefferman square function estimate and was proved by Biggs--Brandes--Hughes \cite{BiggsBrandesHughes2022}:
\begin{proposition}
Let $F \in \mathcal{S}(\R^4)$ with $\text{supp}(\hat{F}) \subseteq \mathcal{N}_\delta(\gamma_{2,\C})$. Then the following estimate holds:
\begin{equation*}
\| F \|_{L^4(\R^4)} \lesssim \big\| \big( \sum_{\theta \in \Theta_\delta} |F_\theta|^2 \big)^{\frac{1}{2}} \big\|_{L^4(\R^4)}.
\end{equation*}
\end{proposition}

\smallskip

We turn to the complex planks on the scale $R$ for $\C \Gamma_2$. Let $\mathcal{R}_{\C}(R^{-1}) = R^{-\frac{1}{2}} (\Z \times \Z) \cap [-1,1]^2$\footnote{The base points of $\mathcal{R}_{\C}(\delta)$ are corresponding to base points for a covering of the $\delta$-neighbourhood of $\gamma_{2,\C}$, like in Section \ref{section:SquareFunctionsCones}.}, which can be regarded as a subset of $\C$ as well.

Let $z= (s,t) \in \mathcal{R}_{\C}(R^{-1})$. The central line is given by
\begin{equation*}
\mathbf{c}(s,t)= (s,t,s^2-t^2,2st,1).
\end{equation*}
We define the (real) tangential vectors as
\begin{equation*}
\mathbf{t}_s(s,t) = (1,0,2s,2t,0), \quad \mathbf{t}_t(s,t) = (0,1,-2t,2s,0)
\end{equation*}
and the normal vectors as
\begin{equation*}
\mathbf{n}_s(s,t) = (-2s,2t,1,0,0), \quad \mathbf{n}_t(s,t) = (-2t,-2s,0,1,0), \quad \mathbf{n}_5 = (0,0,0,0,1).
\end{equation*}
Note that the complex derivative of the base curve $z \mapsto (z,z^2) \in \C$ is given by
$\dot{\gamma}_{2,\C}(z) = (1,2z) \in \C^2$. From this we can read off the tangential vectors as
\begin{equation*}
\begin{split}
\mathbf{t}_s(s,t) &= (\Re \dot{\gamma}_{2,\C}(z)_1, \Im \dot{\gamma}_{2,\C}(z)_1, \Re \dot{\gamma}_{2,\C}(z)_2, \Im \dot{\gamma}_{2,\C}(z)_2), \\
\mathbf{t}_t(s,t) &= (-\Im \dot{\gamma}_{2,\C}(z)_1, \Re \dot{\gamma}_{2,\C}(z)_1, - \Im \dot{\gamma}_{2,\C}(z)_2, \Re \dot{\gamma}_{2,\C}(z)_2).
\end{split}
\end{equation*}
Let $F: \C^2 \to \R^4$ denote the identification $(z_1,z_2) \mapsto (\Re z_1, \Im z_1, \Re z_2, \Im z_2)$. We clarify how linear combinations of $\mathbf{t}_s$ and $\mathbf{t}_t$ correspond to complex multiples of $(1,2z) \in \C^2$. Compute for $c+id \in \R + i \R$, $(1,2z) \in \C^2$:
\begin{equation*}
F((c+id)(1,2z)) = c F(1,2z) + d F(i(1,2z)) = c \mathbf{t}_s(z) + d \mathbf{t}_t(z).
\end{equation*}
This will be very useful to perceive the linear combinations $a \mathbf{t}_s(s,t) + b \mathbf{t}_t(s,t)$ for $|a|, |b| \lesssim M$ as complex multiple $\lambda \cdot (1,2z)$ with $\lambda \in \C$, $|\lambda| \lesssim M$. Further, we let
\begin{equation*}
\mathbf{n}_{\C}(z) = (-2 \bar{z},1) \in \C^2,
\end{equation*}
from which the real normal vectors can be read off: $\mathbf{n}_s = F(\mathbf{n}_{\C})$ and $\mathbf{n}_t  = F(i \mathbf{n}_{\C})$.

\smallskip

Clearly, $\{\mathbf{t}_s,\mathbf{t}_t,\mathbf{n}_x \}$, $x \in \{s,t,5\}$ are orthogonal and  we cover $\mathcal{N}_{R^{-1}}(\C \Gamma_2)$ with $\theta_z$, $z = s + it \in \mathcal{R}_{\C}(R^{-1})$ defined as follows:
\begin{equation*}
\begin{split}
\theta_z &= \{ a \mathbf{c}(s,t) + b_1 \mathbf{t}_s(s,t) + b_2 \mathbf{t}_t(s,t) + c_1 \mathbf{n}_s(s,t) + c_2 \mathbf{n}_t(s,t) + e \mathbf{n}_5 : \\
&\quad \frac{1}{2} \leq a \leq 1, \; b_i, c_i, e \in \R, \, |b_i| \leq d R^{-\frac{1}{2}}, \; |c_i|, |e| \leq d R^{-1}  \}
\end{split}
\end{equation*}
for some $d > 1$.

\smallskip

Let $\Theta_{R^{-1}} = \{ \theta_z : z \in \mathcal{R}_{\C}(R^{-1}) \}$. We recall the statement of Theorem \ref{thm:SquareFunctionComplexCone}:
\begin{theorem}
Let $F \in \mathcal{S}(\R^5)$ with $\text{supp}(\hat{F}) \subseteq \mathcal{N}_{R^{-1}}(\C \Gamma_2)$. Then the following estimate holds:
\begin{equation*}
\| F \|_{L^4(\R^5)} \lesssim_\varepsilon R^\varepsilon \big\| \big( \sum_{\theta \in \Theta_{R^{-1}}} |F_{\theta}|^2 \big)^{\frac{1}{2}} \big\|_{L^4(\R^5)}.
\end{equation*}
\end{theorem}

The proof follows the roadmap from \cite{GuthWangZhang2020} with the crucial incidence estimates proved via the local arguments from Section \ref{section:SquareFunctionsCones}.

We turn to the notions involved in formulating the Kakeya estimate: In the following we consider $f \in \mathcal{S}(\R^5)$ with $\text{supp}(\hat{f}) \subseteq \mathcal{N}_{r^{-2}}(\C \Gamma_2)$ for notational convenience and for $\tau \in \Theta_{s^2}$, $1 \leq s \leq r^{-1}$ we consider
\begin{equation*} 
U_{\tau,r^2} = \text{conv} \big( \bigcup_{\theta \subseteq \tau} \theta^* \big),
\end{equation*}
and let for $U \parallel U_{\tau,r^2}$:
\begin{equation*}
S_U f = \big( \sum_{\theta \subseteq \tau} |f_{\theta}|^2 \big)^{\frac{1}{2}} \vert_U.
\end{equation*}

We consider the complexification of centred planks. 
For $\sigma \in [r^{-1},1] \cap 2^{\Z}$, we define for $z \in \mathcal{R}_{\C}(r^{-2} \sigma^{-2})$ the centred plank
\begin{equation*}
\begin{split}
\Theta(\sigma,z) &= \{ a \mathbf{c}(s,t) + b_1 \mathbf{t}_s(s,t) + b_2 \mathbf{t}_t(s,t) + c_1 \mathbf{n}_s(s,t) + c_2 \mathbf{n}_t(s,t) + e \mathbf{n}_5 : \\
&\quad - \sigma^2 \leq a \leq \sigma^2, \; b_i, c_i, e \in \R, \, |b_i| \leq D r^{-1} \sigma, \; |c_i|, |e| \leq D r^{-2}  \}
\end{split}
\end{equation*}
for some $D \gg d$ (say $d=3$, $D=30$).

\smallskip

Like above, we are guided by the idea that at the height $h \sim \sigma^2$ the intersection $\bigcup_z \Theta(\sigma,z) \cap \{ \omega_5 = h \}$ is supposed to canonically cover the $r^{-2}$-neighbourhood of the complex curve $h \gamma_{2,\C}(z)$.

\smallskip

We denote the collection of centred planks by
\begin{equation*}
\mathbf{CP}_\sigma = \{ \Theta(\sigma,z) \, : \, z \in \mathcal{R}_{\C}(r^{-2} \sigma^{-2} )\}
\end{equation*}
and, analogous to the previous section, for $2^{\Z} \ni \sigma > r^{-1}$:
\begin{equation*}
\Omega_{\sigma} = \bigcup \mathbf{CP}_{\sigma} \backslash \bigcup \mathbf{CP}_{\sigma/2},
\end{equation*}
and $\Omega_{r^{-1}} = \bigcup \mathbf{CP}_{r^{-1}}$. With $\mathbf{CP}_1$ corresponding to $\Theta_{r^{-2}}$, we have
\begin{equation}
\label{eq:DecompositionFourierSupportComplex}
\bigcup_{z \in \mathcal{R}_{\C}(r^{-2})} \tilde{\theta}(z) \subseteq \bigcup_{\sigma \in [r^{-1},1]} \Omega_{\sigma}.
\end{equation}

\subsection{A Kakeya estimate for the complex cone}

We can now formulate the Kakeya estimate for the complex cone:
\begin{proposition}
\label{prop:KakeyaComplexCone}
Let $r \gg 1$. The following estimate holds:
\begin{equation}
\label{eq:KakeyaComplexCone}
\int_{\R^5} \big| \sum_{\theta \in \Theta_{r^{-2}}} |f_{\theta}|^2 \big|^2 \lesssim \sum_{r^{-1} \leq s \leq 1} \sum_{\tau \in \Theta_{s^2}} \sum_{U \parallel U_{\tau,r^2}} |U|^{-1} \| S_U f \|^4_{L^2(U)}.
\end{equation}
\end{proposition}
\begin{proof}
We use Plancherel's theorem and the decomposition \eqref{eq:DecompositionFourierSupportComplex} to write
\begin{equation*}
\int_{\R^5} \big| \sum_{\theta \in \Theta_{r^{-2}}} |f_{\theta}|^2 \big|^2 \leq \sum_{r^{-1} \leq \sigma \leq 1} \int_{\Omega_{\sigma}} \big| \sum_{\theta \in \Theta_{r^{-2}}} (|f_\theta|^2) \widehat{\,} (\omega) \big|^2.
\end{equation*}
Given $\sigma$, we associate the sectors $\theta_z \in \Theta_{r^{-2}}$ to centred planks according to the complex distance: For $z \in \mathcal{R}_{\C}(r^{-2})$ let $z' \in \mathcal{R}_{\C}(r^{-2} \sigma^{-2})$ with $|z-z'| \leq 2 r^{-1} \sigma^{-1}$ and write $\theta_z \in \Theta(\sigma,z')$.

\smallskip

It turns out that after changing to the complex description, the arguments from Section \ref{section:SquareFunctionsCones} can be applied to obtain the desired almost orthogonal decomposition at scale $\sigma$. Write for $\omega \in \Omega_\sigma$:
\begin{equation*}
\big| \sum_{\theta \in \Theta_{r^{-2}}} (|f_\theta|^2) \widehat{\,} (\omega) \big|^2 = \big| \sum_{\Theta(\sigma,z) \in \mathbf{CP}_{\sigma}} \sum_{\theta \in \Theta(\sigma,z)} \big( |f_{\theta}|^2 \big) \widehat{\,} (\omega) \big|^2,
\end{equation*}
and we shall show the finite overlap for $\Theta(\sigma,z)$:
\begin{equation}
\label{eq:FiniteOverlapComplex}
\big| \sum_{\Theta(\sigma,z) \in \mathbf{CP}_{\sigma}} \sum_{\theta \in \Theta(\sigma,z)} \big( |f_{\theta}|^2 \big) \widehat{\,}(\omega) \big|^2 \lesssim \sum_{\Theta(\sigma,z) \in \mathbf{CP}_{\sigma}} \big| \sum_{\theta \in \Theta(\sigma,z)} \big( |f_{\theta}|^2 \big) \widehat{\,}(\omega) \big|^2.
\end{equation}

In the first step to prove \eqref{eq:FiniteOverlapComplex} we show the following lemma:
\begin{lemma}
Let $\omega \in \Omega_{\sigma} \cap \tilde{\theta}(z)$. Then there is $\Theta(\sigma,z') \in \mathbf{CP}_{\sigma}$ with $\omega \in \Theta(\sigma,z)$ and $|z-z'| \leq 4 r^{-1} \sigma^{-1}$.
\end{lemma}

To this end, we employ the local analysis from the previous section. We analyze the set $\Omega_{\sigma} \cap \{ \omega_5 = h \}$ as before. For $\theta = \theta(z)$ with $z \in \mathcal{R}_{\C}(r^{-2})$ we can regard $\tilde{\theta}(z) \cap \{ \omega_5 = h \}$ as a complex rectangle centered at $h \gamma_{2,\C}(z)$ with complex length $r^{-1}$ into the tangential direction and $r^{-2}$ into the normal direction. Identifying $\R^4 \equiv \C^2$ we find
\begin{equation*}
\begin{split}
&\quad \pi_{12,\C}(\text{supp}(\mathcal{F}(|f_\theta|^2)) \cap \{ \omega_5 = h \}) \\
 &= \{ h \gamma_{2,\C}(z) + \ell \dot{\gamma}_{2,\C}(z) + c r^{-2} \mathbf{n}_{\C}(z) \, : \ell \in \C, c \in \C, |\ell| \leq d r^{-1}, \; |c| \leq d \}.
\end{split}
\end{equation*}
By $\pi_{12,\C}: \R^5 \equiv \C^2 \times \R \to \C^2$ we denote the projection onto the first two complex coordinates.

We have the following complex variant of Lemma \ref{lem:ComparabilityTheta}:
\begin{lemma}
Let $\sigma \in [r^{-1},1] \cap 2^{\Z}$, $h \leq \sigma^2$, and $|z| \leq 1/8$. Let
\begin{equation*}
p = h \gamma_{2,\C}(z) + \ell \dot{\gamma}_{2,\C}(z) + c r^{-2} \mathbf{n}_{\C}(z)
\end{equation*}
with $|\ell| \lesssim r^{-1} \sigma$, and $|c| \leq d$. Then we have
\begin{equation*}
p = h \gamma_{2,\C}(z') + \ell_1 \dot{\gamma}_{2,\C}(z') + C r^{-2} \mathbf{n}_{\C}(z')
\end{equation*}
for some $z' \in \mathcal{R}_{\C}(r^{-2} \sigma^{-2})$ with $|z-z'| \leq 2 r^{-1} \sigma^{-2}$, $\ell_1,C \in \C$ with $|\ell_1| \sim |\ell|$, and $|C| \leq C^*(d)$.
\end{lemma}
The lemma is proved like its real counterpart for $k=2$ via Taylor expansion, apart from the formal difference that the Taylor expansion is carried out in the complex plane. We omit the details to avoid repetition.

\smallskip

Next, we show the following extension of Lemma \ref{lem:TwoEndsRepresentation}:
\begin{lemma}
Let $h \leq \sigma^2$, $p \in \tilde{\theta}(z) \cap \{ \omega_5 = h \} \cap \Omega_{\sigma}$ and $\sigma \gg r^{-1}$. We have the representation
\begin{equation}
\label{eq:Rep(P)Complex}
p = h \gamma_{2,\C}(z') + \ell \dot{\gamma}_{2,\C}(z') + C r^{-2} \mathbf{n}_{\C}(z')
\end{equation}
with $z' \in \mathcal{R}_{\C}(r^{-2} \sigma^{-2})$, $|z-z'| \leq 2 r^{-1} \sigma^{-1}$, $|\ell| \lesssim r^{-1} \sigma$, and $C \in \C$ with $|C| \leq C^*$.

For $h \ll \sigma^2$ \eqref{eq:Rep(P)Complex} holds with $|\ell| \sim r^{-1} \sigma$.
\end{lemma}
\begin{proof}
Suppose we have the representation
\begin{equation}
\label{eq:Rep(P)DifferentLength}
p = h \gamma_{2,\C}(z) + \ell \dot{\gamma}_{2,\C}(z) + c r^{-2} \mathbf{n}_{\C}(z) = h \gamma_{2,\C}(z') + \ell_1 \dot{\gamma}_{2,\C}(z') + C r^{-2} \mathbf{n}_{\C}(z')
\end{equation}
for $p \in \Omega_\sigma \cap \{ \omega_5 = h \}$ with $|\ell| \gg r^{-1} \sigma$ and $|\ell_1| \lesssim r^{-1} \sigma$, $|c| \ll |C| \lesssim 1$.

\smallskip

We let $z' = z + \Delta z$. By Taylor expansion we find
\begin{equation}
\label{eq:TaylorExpansionComplex}
\begin{split}
&\quad h \gamma_{2,\C}(z + \Delta z) + \ell_1 \dot{\gamma}_{2,\C}(z + \Delta z) + C r^{-2} \mathbf{n}_{\C}(z + \Delta z) \\
&= h \gamma_{2,\C}(z) + h \Delta z \dot{\gamma}_{2,\C}(z) + \frac{h (\Delta z)^2}{2} \ddot{\gamma}_{2,\C}(z) + \ell_1 ( \dot{\gamma}_{2,\C}(z) + \Delta z \ddot{\gamma}_{2,\C}(z)) \\
&\quad + C r^{-2} \mathbf{n}_{\C}(z+\Delta z).
\end{split}
\end{equation}

Plugging \eqref{eq:TaylorExpansionComplex} into \eqref{eq:Rep(P)DifferentLength} and separating the first and second complex coordinate, this yields the conditions
\begin{equation*}
\left\{ \begin{array}{cl}
\ell &= \ell_1 + h \Delta z + \mathcal{O}(r^{-2}), \\
\big( \frac{h \Delta z}{2} + \ell_1 \big) \Delta z &= \mathcal{O}(r^{-2}).
\end{array} \right.
\end{equation*}

Since by assumption $|\ell| \gg |\ell_1|$, we have $|\ell| \sim h |\Delta z|$. The second identity implies $h |\Delta z|^2 = \mathcal{O}(r^{-2})$. If $h \sim \sigma^2$, then $|\Delta z | \lesssim r^{-1} \sigma^{-1}$. This in turn implies by the first identity
\begin{equation*}
|\ell| \lesssim |\ell_1| + h |\Delta z| + \mathcal{O}(r^{-2}) \lesssim r^{-1} \sigma,
\end{equation*}
which settles the case $h \sim \sigma^2$.

\smallskip

In the following suppose that $h \ll \sigma^2$. In this case we obtain from a Taylor expansion
\begin{equation*}
\begin{split}
h \gamma_{2,\C}(z + \frac{\ell}{h}) &= h \gamma_{2,\C}(z) + \ell \dot{\gamma}_{2,\C}(z) + \frac{h \ell^2}{2 h^2} (0,2) \\
&= h \gamma_{2,\C}(z) + \ell \dot{\gamma}_{2,\C}(z) + \mathcal{O}(r^{-2}).
\end{split}
\end{equation*}

This shows that $p$ given by the left hand side of \eqref{eq:Rep(P)DifferentLength} satisfies $p \in \mathcal{N}_{C r^{-2}}(h \gamma_{2,\C})$, but since $\mathcal{N}_{C r^{-2}}(h \gamma_{2,\C}) \subseteq \bigcup \mathbf{CP}_{\sigma/2}$, we obtain $p \notin \Omega_\sigma$, which contradicts our assumption.
\end{proof}

This concludes the first part of the proof of \eqref{eq:FiniteOverlapComplex}:
\begin{equation*}
\big| \sum_{\Theta(\sigma,z) \in \mathbf{CP}_{\sigma}} \sum_{\theta \in \Theta(\sigma,z)} \big( |f_{\theta}|^2 \big) \widehat{\,}(\omega) \big|^2 \lesssim \big| \sum_{\substack{\Theta(\sigma,z) \in \mathbf{CP}_{\sigma}, \\ \omega \in 4 \Theta(\sigma,z)}} \big| \sum_{\theta \in \Theta(\sigma,z)} \big( |f_{\theta}|^2 \big) \widehat{\,}(\omega) \big| \big|^2.
\end{equation*}

We note that for $h \ll \sigma^2$ the set $\Omega_{\sigma} \cap \{ \omega_5 = h \}$ essentially consists of the (complex) ends of $\Theta(\sigma,z) \cap \{ \omega_5 = h \}$, which are defined as follows:
\begin{equation*}
\text{Ends}(\Theta(\sigma,z),h) = \{ h \gamma_{2,\C}(z) + \ell \dot{\gamma}_{2,\C}(z) + C r^{-2} \mathbf{n}_{\C}(z) : \, \ell \in \C, \, |\ell| \sim r^{-1} \sigma, \, |C| \leq C^* \}.
\end{equation*}

To conclude \eqref{eq:FiniteOverlapComplex}, we require the following version of Lemma \ref{lem:AuxOverlapII}:
\begin{lemma}
\label{lem:AuxOverlapIIComplex}
Let $\omega \in \Omega_\sigma$. Then the following estimate holds:
\begin{equation*}
\# \{ \Theta(\sigma,\xi) : \omega \in 10 \Theta(\sigma,\xi) \} \lesssim 1.
\end{equation*}
\end{lemma}

We show the following:
\begin{lemma}
\label{lem:AlternatingEndsComplex}
Let $h \ll \sigma^2$ and $\sigma \gg r^{-1}$, let $X \in \{ \Re, \Im \}$, and $\mu \in \{1,-1 \}$. Let $p \in \Omega_{\sigma} \cap \{ \omega_5 = h \}$ and suppose that for $|\ell| \sim |\ell_1| \sim r^{-1} \sigma$, $|C|,|C_1| \lesssim 1$, and for $\Delta z = z'-z$, $|\Delta z | \gg r^{-1} \sigma^{-1}$, it holds
\begin{equation}
\label{eq:ComplexRep(P)}
h \gamma_{2,\C}(z) + \ell \dot{\gamma}_{2,\C}(z) + C r^{-2} \mathbf{n}_{\C}(z) = h \gamma_{2,\C}(z') + \ell_1 \dot{\gamma}_{2,\C}(z') + C _1r^{-2} \mathbf{n}_{\C}(z').
\end{equation}

\smallskip

Then we have
\begin{equation*}
X (\ell) \sim \mu r^{-1} \sigma \Rightarrow X(\ell_1 )  \sim - \mu r^{-1} \sigma.
\end{equation*}
\end{lemma}
\begin{proof}
We carry out a Taylor expansion of the right hand side of \eqref{eq:ComplexRep(P)} to obtain:
\begin{equation*}
0 = ( h \Delta z - \ell + \ell_1) \dot{\gamma}_{2,\C}(z) + ( \frac{h (\Delta z)^2}{2} + \ell_1 \Delta z) \ddot{\gamma}_{2,\C}(z) + \mathcal{O}(r^{-2}).
\end{equation*}
Taking the exterior product with $\dot{\gamma}_{2,\C}(z)$ or $\ddot{\gamma}_{2,\C}(z)$, we find
\begin{equation*}
\left\{ \begin{array}{cl}
h \Delta z - \ell + \ell_1 &= \mathcal{O}(r^{-2}), \\
(h (\Delta z)/2 + \ell_1 ) \Delta z &= \mathcal{O}(r^{-2}).
\end{array} \right.
\end{equation*}
The second identity gives by the minimum size assumption on $\Delta z$ that
\begin{equation*}
|h (\Delta z)/2 + \ell_1| \ll r^{-1} \sigma.
\end{equation*}
Plugging this into the first identity, we find
\begin{equation*}
|h (\Delta z)/2 - \ell| \ll r^{-1} \sigma
\end{equation*}
Consequently,
\begin{equation*}
|X(h (\Delta z)/2 - \ell) | \ll r^{-1} \sigma, \quad |X( h (\Delta z)/2 + \ell_1)| \ll r^{-1} \sigma.
\end{equation*}
This necessitates $X(\ell) \sim - X(\ell_1)$, which completes the proof.
\end{proof}

We can now conclude the proof of Lemma \ref{lem:AuxOverlapIIComplex}:

\begin{proof}[Proof~of~Lemma~\ref{lem:AuxOverlapIIComplex}]
Suppose that $h \sim \sigma^2$. Since $\mathbf{CP}_{\sigma} \cap \{ \omega_5 = h \}$ (after projection to $\C^2$) forms a canonical covering of the $r^{-2}$-neighbourhood of $h \gamma_{2,\C}$, the finite overlap is immediate.

\smallskip

We turn to $h \ll \sigma^2$: In this case it suffices to analyze the overlap of the ends of $\Theta(\sigma,\xi$. Since there are only finitely many $z' \in \mathcal{R}_{\C}(r^{-2} \sigma^{-2})$ such that $|z-z'| \lesssim r^{-1} \sigma^{-1}$ it suffices to check the overlap of the ends $\Theta(\sigma,z)$ and $\Theta(\sigma,z')$ with $|z-z'| \gg r^{-1} \sigma^{-1}$. So, suppose that $\pi_{12}(\omega) = p \in \text{End}(\Theta(\sigma,z),h) \cap \text{End}(\Theta(\sigma,z'),h)$ like in \eqref{eq:ComplexRep(P)}.
We can invoke Lemma \ref{lem:AlternatingEndsComplex} to find that for $X(\ell) \sim \mu r^{-1} \sigma$ with $\mu \in \{1,-1\}$ it holds $X(\ell_1) \sim - \mu r^{-1} \sigma$. Now suppose that there is a third $z'' \in \mathcal{R}_{\C}(r^{-2} \sigma^{-2})$ for which $p \in \text{Ends}(\Theta(\sigma,z''),h)$ such that we have the representations
\begin{equation*}
h \gamma_{2,\C}(z) + \ell \dot{\gamma}_{2,\C}(z) + C r^{-2} \mathbf{n}_{\C}(z) = h \gamma_{2,\C}(z'') + \ell_2 \dot{\gamma}_{2,\C}(z'') + C_2 r^{-2} \mathbf{n}_{\C}(z''),
\end{equation*}
and
\begin{equation*}
h \gamma_{2,\C}(z') + \ell_1 \dot{\gamma}_{2,\C}(z') + C_1 r^{-2} \mathbf{n}_{\C}(z') = h \gamma_{2,\C}(z'') + \ell_2 \dot{\gamma}_{2,\C}(z'') + C _2 r^{-2} \mathbf{n}_{\C}(z'').
\end{equation*}
Then two applications of Lemma \ref{lem:AlternatingEndsComplex} yield that
\begin{equation*}
|z'' - z | \lesssim r^{-1} \sigma^{-1} \text{ or } |z'' - z'| \lesssim r^{-1} \sigma^{-1}.
\end{equation*}
So, $z''$ must be either neighboring $z$ or $z'$, which completes the proof.
\end{proof}

Now we can conclude the proof of the Kakeya estimate for the complex cone:

\medskip

\emph{Conclusion~of~the~Proof~of~Proposition~\ref{prop:KakeyaComplexCone}}: We have proved \eqref{eq:FiniteOverlapComplex}, which yields
\begin{equation*}
\begin{split}
\int_{\Omega_{\sigma}} \big| \sum_{\theta \in \Theta_{r^{-2}}} |f_{\theta}|^2 \big|^2 &= \int_{\Omega_{\sigma}} \big| \sum_{\Theta(\sigma,z) \in \mathbf{CP}_{\sigma}} \sum_{\theta \in \Theta(\sigma,z)} |f_{\theta}|^2 \big|^2 \\
&\lesssim \int_{\Omega_{\sigma}} \sum_{\Theta(\sigma,z) \in \mathbf{CP}_{\sigma}} \big| \sum_{\theta \in \Theta(\sigma,z)} |f_{\theta}|^2 \big|^2.
\end{split}
\end{equation*}
This yields
\begin{equation*}
\int_{\R^5} \big| \sum_{\theta \in \Theta_{r^{-2}}} |f_{\theta}|^2 \big|^2 \lesssim \sum_{\sigma \in [r^{-1},1]} \sum_{\Theta(\sigma,z) \in \mathbf{CP}_{\sigma}} \int_{\R^5} \big| \sum_{\theta \in \Theta(\sigma,z)} |f_{\theta}|^2 \big|^2,
\end{equation*}
from which the right hand side of \eqref{eq:KakeyaComplexCone} follows by the same means as above, i.e., using the essentially constant property. We omit the details and refer to Section \ref{section:SquareFunctionsCones} to avoid repetition.
\end{proof}

\subsection{Induction-on-scales}

In this section we indicate how the Kakeya estimate from Proposition \ref{prop:KakeyaComplexCone} implies the square function estimate. This closely follows the roadmap from \cite{GuthWangZhang2020} and for the sake of brevity, we shall focus on defining and showing the analogs and extensions of the estimates used to carry out the induction-on-scales. With these at hand, the proof can be concluded like in \cite{GuthWangZhang2020}.

\smallskip

For $ r \geq 1$ let $\mathcal{U}_r$ denote a finitely overlapping covering of $\R^5$ with $r$-balls. Let $f \in \mathcal{S}(\R^5)$ with $\text{supp}(\hat{f}) \subseteq \mathcal{N}_{R^{-1}}(\C \Gamma_2)$. For $1 \leq r \leq R$ define the two-scale-quantity $S(r,R)$ as infimum over $C \geq 1$ such that:
\begin{equation}
\label{eq:DefinitionSComplex}
\sum_{B_r \in \mathcal{U}_r} |B_r|^{-1} \| S_{B_r} f \|^4_{L^2(B_r)} \leq C \sum_{R^{-1} \leq \sigma \leq 1} \sum_{\tau \in \Theta_{\sigma}} \sum_{U \parallel U_{\tau,R}} |U|^{-1} \| S_U f \|^4_{L^2(U)}.
\end{equation}
Like in \cite{GuthWangZhang2020}, Theorem \ref{thm:SquareFunctionComplexCone} is a consequence of $S(1,R) \lesssim_\varepsilon R^\varepsilon$.

We formulate the Kakeya estimate in terms of the two-scale-quantity:
\begin{proposition}
\label{prop:ComplexKakeyaTwoScale}
Let $1 \leq r \leq R \leq r^2$. Then the following estimate holds:
\begin{equation*}
S(r,R) \lesssim C.
\end{equation*}
\end{proposition}

\smallskip

For $I_K \subseteq [1/2,1]$ a closed interval with length $1/K$ we define
\begin{equation*}
\C \Gamma_{2,K} = \{ (z,z^2/h,h) \in \C^2 \times \R \, : \, |z| \leq 1, \; h \in I_K \}.
\end{equation*}

For $R \gg 1$ we shall eventually choose $K = R^\delta$. For $f \in \mathcal{S}(\R^5)$ with $\text{supp}(\hat{f}) \subseteq \mathcal{N}_{R^{-1}}(\C \Gamma_2)$ we consider the quantity $S_K(r,R)$ defined as infimum over $C$ such that \eqref{eq:DefinitionSComplex} holds.

On the one hand, we have by decomposing the height into $1/K$-intervals:
\begin{equation}
\label{eq:SlicingComplexCone}
S(r,R) \lesssim K S_K(r,R).
\end{equation}
Secondly, the introduction of $S_K$ allows us to jump start the induction-on-scales by noting that for $\text{supp}(\hat{f}) \subseteq \mathcal{N}_{R^{-1}}(\C \Gamma_{2,K})$ we have
\begin{equation}
\label{eq:SquareFunctionEstimateComplexSlice}
\big( \int_{\R^5} |f|^4 \big)^{\frac{1}{4}} \lesssim \big( \int_{\R^5} \big( \sum_{\theta \in \Theta_{K^{-1}}} |f_{\theta}|^2 \big)^2 \big)^{\frac{1}{4}}.
\end{equation}
\begin{proof}
Writing
\begin{equation*}
\int_{\R^5} |f|^4 = \int_{\R^5} \big( \sum_{\theta_1 \in \Theta_{K^{-1}}} f_{\theta_1} \cdot \sum_{\theta_2 \in \Theta_{K^{-1}}} f_{\theta_2} \big) \overline{\big( \sum_{\theta_3 \in \Theta_{K^{-1}}} f_{\theta_3} \cdot \sum_{\theta_4 \in \Theta_{K^{-1}}} f_{\theta_4} \big)}.
\end{equation*}
and applying Plancherel's theorem, like in the proof of Theorem \ref{thm:DegenerateSquareFunctionEstimateCurves} we are led to analyzing solutions to the system for $h_i \in I_K$ and $(z_i,z_i^2/h_i,h_i) \in \theta_i$:
\begin{equation*}
\left\{ \begin{array}{cl}
z_1 + z_2 &= z_3 + z_4, \\
\frac{z_1^2}{h_1} + \frac{z_2^2}{h_2} &= \frac{z_3^2}{h_3} + \frac{z_4^2}{h_4} + \mathcal{O}(K^{-1}).
\end{array} \right.
\end{equation*}

Since $h_i \in I_K$, the above implies
\begin{equation*}
\left\{ \begin{array}{cl}
z_1 + z_2 &= z_3 + z_4, \\
z_1^2 + z_2^2 &= z_3^2 + z_4^2 + \mathcal{O}(K^{-1}),
\end{array} \right.
\end{equation*}
and now, the biorthogonality follows from the complex extension of the C\'ordoba--Fefferman square function estimate (see \cite{BiggsBrandesHughes2022}). With the biorthogonality at hand, \eqref{eq:SquareFunctionEstimateComplexSlice} follows from applying the Cauchy-Schwarz inequality.
\end{proof}
Combining \eqref{eq:SquareFunctionEstimateComplexSlice} with Proposition \ref{prop:ComplexKakeyaTwoScale}, we record the following lemma.
\begin{lemma}
\label{lem:StartInductionScalesComplex}
The following estimate holds:
\begin{equation*}
S_K(1,K) \lesssim 1.
\end{equation*}
\end{lemma}

The final ingredient in the induction-on-scales from \cite{GuthWangZhang2020} is the Lorentz rescaling, which enlarges small sectors from the real cone to the full cone. Presently, we map $d$-sectors $\tau \subseteq \C \Gamma_2$ defined by
\begin{equation}
\label{eq:ComplexdSector}
\tau = \{ (z,\zeta = z^2 /h, h) : \, h \in [1/2,1], \; \big| \frac{z}{h} - \nu \big| \leq d \},
\end{equation}
to the full cone $\C \Gamma_2$.

\smallskip

The complex generalization is straight-forward, and we record its properties in the following:
\begin{lemma}
Let $\tau \subseteq \C \Gamma_2$ be a sector centered at $\nu \in \C$, $|\nu| \leq 1$ with aperture $d$ given by \eqref{eq:ComplexdSector}. There is a linear transformation $(z,\zeta,h) \mapsto (z',\zeta',h')$ given by
\begin{equation*}
z' = d^{-1}(z - \nu h), \quad \zeta' = d^{-2} \zeta + \text{Lin}(z',h'), \quad h' = h,
\end{equation*}
which extends to a map $\mathcal{N}_{R^{-1}}(\tau) \to \mathcal{N}_{d^{-2} R^{-1}}(\C \Gamma_2)$ and establishes a correspondence between $\Theta_{R^{-1}} \ni \theta \subseteq \mathcal{N}_{R^{-1}}(\tau) $ and $\theta' \in \Theta_{d^{-2} R^{-1}}$.
\end{lemma}

With the Lorentz rescaling at disposal, we record the following analog of \cite[Lemma~]{GuthWangZhang2020}:
\begin{lemma}
\label{lem:RescalingComplexCone}
For any $1 \leq r_1 < r_2 \leq r_3$, it holds
\begin{equation*}
S_K(r_1,r_3) \leq \log(r_2) S_K(r_1,r_2) \max_{s \in [r_2^{-\frac{1}{2}},1]} S_K(s^2 r_2, s^2 r_3).
\end{equation*}
\end{lemma}

\begin{proof}[Proof~of~Theorem~\ref{thm:SquareFunctionComplexCone}]
Like in \cite{GuthWangZhang2020} it suffices to show for $1\leq r < R$
\begin{equation*}
S(r,R) \leq C_\varepsilon \big( \frac{R}{r} \big)^\varepsilon.
\end{equation*}
Choosing $K=R^\delta$, by \eqref{eq:SlicingComplexCone} it is enough to show
\begin{equation*}
S_K(r,R) \leq C_\varepsilon \big( \frac{R}{r} \big)^\varepsilon.
\end{equation*}

Taking Lemmas \ref{lem:StartInductionScalesComplex}, Proposition \ref{prop:ComplexKakeyaTwoScale}, and Lemma \ref{lem:RescalingComplexCone} together, this follows from the same arguments as in the proof of \cite[Proposition~3.4]{GuthWangZhang2020}. This finishes the proof of Theorem \ref{thm:SquareFunctionComplexCone}.
\end{proof}

\section*{Acknowledgements}
Financial support from the Humboldt foundation (Feodor-Lynen fellowship) and partial support by the NSF grant DMS-2054975 is gratefully acknowledged.

\bibliographystyle{plain}

\end{document}